\documentclass[11pt,reqno]{amsart}

\usepackage[utf8]{inputenc}

\usepackage{amsmath,amsthm,amssymb}
\usepackage{amssymb}

\usepackage{graphics}
\usepackage{hyperref}
\usepackage[usenames, dvipsnames]{xcolor}
\usepackage{soul}

\usepackage{mathtools}
\mathtoolsset{showonlyrefs}

\usepackage[square,sort,comma,numbers]{natbib}

\usepackage{verbatim}
\usepackage{longtable}

\definecolor{darkblue}{rgb}{0.0,0.0,0.3}
\hypersetup{colorlinks,breaklinks,
  linkcolor=darkblue,urlcolor=darkblue,
anchorcolor=darkblue,citecolor=darkblue}

\theoremstyle{plain}
\newtheorem{theorem}{Theorem}[section]
\newtheorem*{theorem*}{Theorem}
\newtheorem{lemma}[theorem]{Lemma}
\newtheorem{proposition}[theorem]{Proposition}
\newtheorem*{proposition*}{Proposition}
\newtheorem{corollary}[theorem]{Corollary}
\newtheorem*{corollary*}{Corollary}

\theoremstyle{definition}
\newtheorem{remark}[theorem]{Remark}

\numberwithin{equation}{section}

\renewcommand{\Im}{\operatorname{Im}}
\renewcommand{\Re}{\operatorname{Re}}

\DeclareMathOperator*{\Res}{Res}

\DeclareRobustCommand{\mhl}[1]{%
  \ifmmode\text{\hl{$#1$}}\else\hl{#1}\fi
}

\newcommand{\legendre}[2]{\genfrac(){}{}{#1}{#2}}
\newcommand{\slegendre}[2]{\genfrac(){}{1}{#1}{#2}}
\newcommand{\sumstar}{\sideset{}{^\star}\sum}

\title{Sums of Cusp Form Coefficients Along Quadratic Sequences}
\author[Chan Ieong Kuan, David Lowry-Duda, and Alexander Walker]{Chan Ieong Kuan, David Lowry-Duda,
and Alexander Walker, with an Appendix by Raphael S. Steiner}

\begin{document}

\begin{abstract}
Let $f(z) = \sum A(n) n^{(k-1)/2} e(nz)$ be a cusp form of weight $k \geq 3$ on $\Gamma_0(N)$ with character $\chi$. By studying a certain shifted convolution sum, we prove that $\sum_{n \leq X} A(n^2+h) = c_{f,h} X + O_{f,h,\epsilon}(X^{\frac{3}{4}+\epsilon})$ for $\epsilon>0$, which improves a result of Blomer~\cite{Blomer08} with error $X^{\frac{6}{7}+\epsilon}$.
\end{abstract}

\maketitle

\section{Introduction}

In~\cite{Hooley63}, Hooley considers the average behavior of the divisor function $d(n)$ within a quadratic sequence and proves that
\begin{align} \label{eq:divisor-quadratic}
  S(X):= \sum_{n \leq X} d(n^2+h) = c_h X \log X + c_h' X + O_{h,\epsilon}\big(X^{\frac{8}{9}}\log^3 X\big)
\end{align}
for constants $c_h, c_h'$ (when $-h$ is non-square) using the theory of exponential sums.
Hooley's error term was improved by Bykovskii~\cite{Bykovskii87}, who uses the spectral theory of automorphic forms to study the generalized sum $\sum_{n \leq X} \sigma_\nu(n^2 + h)$, in which $\sigma_\nu(n) = \sum_{d \mid n} d^\nu$.
In the case $\nu =0$, Bykovskii obtains $S(X) = c_h X \log X + c_h' X + O(X^{2/3+\epsilon})$ for any $\epsilon > 0$.

An analogous question for the normalized coefficients of a GL(2) cusp form was introduced by Blomer in~\cite{Blomer08}. Let $f(z) = \sum a(n)e(nz)$ be a cusp form on $S_k(\Gamma_0(N),\chi)$ with weight $k \geq 4$ and set $A(n) = a(n)/n^{(k-1)/2}$.  For any monic quadratic polynomial $q(x) \in \mathbb{Z}[x]$, Blomer proves
\begin{align} \label{eq:blomer-result}
\sum_{n \leq X} A(q(n)) = c_{f,q} X + O_{f,q,\epsilon}\big(X^{\frac{6}{7}+\epsilon}\big)
\end{align}
for some constant $c_{f,q}$ which equals $0$ in most but not all cases.

Both Hooley and Bykovskii rely on the convolution identity $\sigma_0 = 1 * 1$.
As this has no analogue for cusp forms, Blomer instead proceeds by writing $f(z)$ as a sum of Poincar\'e
series, whose $q(n)$-th Fourier coefficients involve sums of the form
\[\sum_{c \geq 1} \frac{1}{N c} S_\chi(m,q(n);Nc) J_{k-1} \Big(\frac{4\pi \sqrt{q(n)m}}{Nc}\Big),\]
in which $J_{k-1}$ is the $J$-Bessel function and $S_\chi$ is a twisted Kloosterman sum. Blomer evaluates a smooth version of the sum over $n \leq X$ using Poisson summation, which converts the sums $S_\chi(m,q(n);Nc)$ into half-integral weight Kloosterman sums. Blomer's result then follows from cancellation in the latter, as proved using a half-integral weight Kuznetsov formula.

Two alternative methods for treating the cusp form analogy are given by Templier and Tsimerman in~\cite{TemplierTsimerman12}. The first is inspired by earlier work by Sarnak on $d(n^2+h)$ in~\cite{sarnak1984additive}, who relates the shifted convolution sum
\[\sum_{n \geq 1} \frac{d(n^2+h)}{(n^2+h)^s}\]
to the Petersson inner product $\langle \Im(z)^{\frac{1}{4}} \theta(z) \overline{E}(z,\frac{1}{2}), P_h(z,s) \rangle$, where $\theta(z)$ is a theta function, $E(z,s)$ is a weight $0$ real analytic Eisenstein series, and $P_h(z,s)$ is a half-integral weight Poincar\'e series. (Sarnak notes that this connects $S(X)$ to the spectrum of the half-integral weight Laplacian, going no further.) In~\cite[\S{4}]{TemplierTsimerman12}, similar methods are applied to study the shifted convolution sum
\begin{align} \label{eq:D_h-definition}
D_h(s) := \sum_{n \geq 0} \frac{r_1(n) a(n+h)}{(n+h)^{s+\frac{k}{2}-\frac{3}{4}}},
\end{align}
in which $r_\ell(n)$ is the number of representations of $n$ as a sum of $\ell$ squares. Templier and
Tsimerman give a meromorphic continuation of $D_h(s)$ and prove that $D_h(s)$ grows polynomially in
$\lvert \Im s \rvert$ in vertical strips. In particular, for any smooth function $g(x)$ on $\mathbb{R}^+$
with Mellin transform $\widetilde{g}(s)$ satisfying $\widetilde{g}(s) \ll \Gamma(s)$ in vertical
strips,~\cite[\S{4.8}]{TemplierTsimerman12} gives a constant $c_{f,h}'$ depending only on $f(z)$ and $h$ for
which
\begin{equation}\label{eq:TT12_statement}
\sum_{n \geq 0} A(n^2+h) g \Big(\frac{n^2+h}{X^2} \Big)
  = c_{f,h}' \,\widetilde{g}(\tfrac{1}{2}) X
  + O_\epsilon\bigg(\frac{X^{\frac{1}{2}+\Theta +\epsilon}}{h^{\frac{1}{2}\Theta - \delta}}\bigg),
\end{equation}
in which $\Theta \leq \frac{7}{64}$ (due to~\cite{KimSarnak03}) denotes progress towards the Selberg
eigenvalue conjecture and $\delta \leq \frac{1}{6}$ (due to~\cite{PetrowYoung19}) denotes progress
towards the Ramanujan--Petersson conjecture for half-integral weight cusp forms.

Templier and Tsimerman give a second proof of their result for $A(n^2+h)$ using representation
theory. This alternative framework allows for equal treatment of holomorphic cusp forms and Maass
cusp forms. Here as before, Templier--Tsimerman restrict to smoothed sums,
so their results cannot be directly compared to the sharp cutoff~\eqref{eq:blomer-result}
from~\cite{Blomer08}.

In this paper, we refine the shifted convolution sum technique described in~\cite[\S{4}]{TemplierTsimerman12} to produce the following sharp cutoff result.

\begin{theorem} \label{thm:main-theorem-intro}
Let $f(z) = \sum a(n) e(nz)$ be a cusp form in $S_k(\Gamma_0(N),\chi)$ and define $A(n) = a(n)/n^{(k-1)/2}$. For $k \geq 3$, $h>0$, and any $\epsilon > 0$, we have
\[
\sum_{n^2 +h \leq X^2} A(n^2+h) = (b_{f,h} + c_{f,h}) X
  +
	O_{f,h,\epsilon}\big(X^{\frac{3}{4} + \epsilon}\big).
\]
The constants $b_{f,h}$ and $c_{f,h}$ are described in~\eqref{eq:discrete-spectrum-residue} and \eqref{eq:residual-general}, respectively.
\end{theorem}

Theorem~\ref{thm:main-theorem-intro} improves the error term
$O(X^{\frac{6}{7}+\epsilon})$ from~\cite{Blomer08}.
As we also show that $D_h(s)$ grows polynomially in $\lvert \Im s \rvert$ in
vertical strips, it would be straightforward ro reprove the smoothed
bound~\eqref{eq:TT12_statement} from our analysis.
We also remark that cusp form analogy still lags behind Bykovskii's
$\frac{2}{3}+\epsilon$ exponent in the divisor function analogue.
Exponents of size $\frac{1}{2}+\epsilon$ are conjectured to hold in both problems.

\section{Outline of Paper}

As in~\cite[\S{4}]{TemplierTsimerman12}, we understand sums of the form $\sum_{n \leq X} A(n^2+h)$ by studying the Dirichlet series $D_h(s)$ defined in~\eqref{eq:D_h-definition}.
In \S{\ref{sec:triple-inner-product}}, we prove Proposition~\ref{prop:D_h-inner-product}, which relates $D_h(s)$ to an inner product involving $f(z)$, the Jacobi theta function, and an appropriate Poincar\'e series $P_h^\kappa(z,s)$. Spectral expansion of $P_h^\kappa(z,s)$ in~\S{\ref{sec:spectral-expansion}} then expresses $D_h(s)$ as a sum of terms corresponding to the discrete, residual, and continuous spectra of the hyperbolic Laplacian.

Our treatment of the discrete spectrum of half-integral weight Maass forms differs greatly from~\cite{TemplierTsimerman12} and represents the main novelty of this work.
We avoid the use of weak estimates for the individual Fourier coefficients $\rho_j(n)$ of Maass forms by exploiting averages over either $n$ or the spectrum of Maass forms.
Our $n$-average appears in~\S{\ref{sec:fourier-coefficient-averages}} and
refines ideas of~\cite[\S{19}]{DFI02} by incorporating uniform bounds for the
Whittaker function. Our spectral average, a refinement of~\cite[Lemma
5]{Blomer08}, appears in~\S{\ref{sec:fourier-coefficient-averages}} and is
proved in Appendix~\ref{sec:appendix}. This appendix is due to Raphael Steiner.

These Fourier coefficient estimates are applied in~\S{\ref{sec:D-estimate}} to prove Theorem~\ref{thm:D-estimate}, a bound for the
sum
\[\mathfrak{D} :=
  \sum_{\vert t_j \vert \sim T}
  \vert \langle y^{\frac{k}{2} + \frac{1}{4}} f \overline{\theta}, \mu_j \rangle \vert^2 e^{\pi \vert t_j \vert},
\]
which averages over an orthonormal basis of (half-integral weight) Maass forms $\mu_j$ with spectral types $\vert t_j \vert \in [T,2T]$.
As in~\cite{Blomer08}, we leverage the fact that $f$ is holomorphic to write it
as a linear combination of holomorphic Poincar\'e series.
Unlike~\cite{Blomer08}, however, these Poincar\'e series are used to form
shifted convolutions, instead of introducing Kloosterman-type sums.
(This was noted as a possible approach in footnote~19 of~\cite{Watkins2019}, but
this wasn't executed there.)

In~\S{\ref{sec:sharp-cutoff}}, we use bounds for $\mathfrak{D}$ to control the
growth of $D_h(s)$ with respect to $\lvert \Im s \rvert$ in vertical strips.
Our main result Theorem~\ref{thm:main-theorem-intro} then follows by Perron's
formula and standard arguments using complex analysis.

\section*{Acknowledgements}

The authors thank Stephen Lester, who introduced them to this problem, as well as Yiannis Petridis, whose suggestions led to a simplification of section~\S{\ref{sec:D-estimate}}.
We also thank Thomas Hulse, whose involvement over many years has improved this paper and moreover its authors.

The first author was supported in part by NSFC (No.\ 11901585).
The second author was supported by the Simons Collaboration in Arithmetic Geometry, Number Theory,
and Computation via the Simons Foundation grant 546235.
The third author was supported by the Additional Funding Programme for Mathematical Sciences, delivered by EPSRC (EP/V521917/1) and the Heilbronn Institute for Mathematical Research.
R.S. would like to extend his gratitude to his employer, the Institute for Mathematical Research
(FIM) at ETH Z\"urich.

\section{A Triple Inner Product}\label{sec:triple-inner-product}

For integral $k \geq 1$ and an even Dirichlet character $\chi$, let
$S_k(\Gamma_0(N), \chi)$ denote the set of cusp forms on $\Gamma_0(N)$
which transform under the character $\chi \cdot \chi_{-1}^k$, where
$\chi_{-1} = (\frac{-1}{\cdot})$.
We assume without loss of generality that $4 \mid N$.
Once and for all, we fix a positive integer $h$ and a weight $k \geq 3$ modular form $f(z) = \sum a(n) e(nz) \in
S_k(\Gamma_0(N),\chi)$.
Here and later, we use the common notation $e(x) := e^{2 \pi i x}$.
Let $\theta(z) = \sum_{n \in \mathbb{Z}} e(n^2 z) = \sum_{n \geq 0} r_1(n) e(nz)$ denote the
classical Jacobi theta function.
The theta function is a modular form of weight $\frac{1}{2}$ on $\Gamma_0(4)$ (see~\cite{Shimura73}),
transforming via
\[
\theta(\gamma z)
= \epsilon_d^{-1} \Big(\frac{c}{d}\Big) (cz+d)^{\frac{1}{2}} \theta(z),
 \quad \gamma = \Big( \begin{matrix} a & b \\ c & d \end {matrix} \Big) \in \Gamma_0(4),
\]
in which $\epsilon_d = 1$ for $d \equiv 1 \bmod 4$, $\epsilon_d =i$ for $d \equiv 3 \bmod 4$, and $(\frac{c}{d})$ denotes the Kronecker symbol.
Then $\upsilon_\theta(\gamma) := \epsilon_d^{-1}(\frac{c}{d})$ is a multiplier system in the sense of~\cite[\S{2}]{Stromberg08}.
Let $P_h^\kappa(z,s)$ denote the weight $\kappa := k-\frac{1}{2}$ twisted Poincar\'e series on $\Gamma_0(N)$, defined by
\begin{align} \label{eq:P_h-definition}
P_h^\kappa(z,s) := \sum_{\gamma \in \Gamma_\infty \backslash \Gamma_0(N)} \overline{\chi(\gamma)} J_\theta(\gamma,z)^{-2\kappa} \Im(\gamma z)^s e(h \gamma z),
\end{align}
in which $J_\theta(\gamma,z) = \upsilon_\theta(\gamma) (cz+d)^\frac{1}{2} / \vert c z + d
\vert^{\frac{1}{2}}$ is the normalized theta cocycle $\theta(\gamma z)/\theta(z)$.
Then $P_h^\kappa(z,w)$ and $\Im(z)^{\frac{k}{2}+\frac{1}{4}} f(z) \overline{\theta(z)}$ transform identically under the action of $\Gamma_0(N)$, so the Petersson inner product $\langle y^{\frac{k}{2}+\frac{1}{4}} f \overline{\theta}, P_h^\kappa(\cdot, \overline{s}) \rangle$ is well-defined over $\Gamma_0(N)$. A standard unfolding argument
gives
\begin{align*}
\langle y^{\frac{k}{2}+\frac{1}{4}} f \overline{\theta}, P_h^\kappa(\cdot, \overline{s}) \rangle
  &= \int_{\Gamma_0(N) \backslash \mathbb{H}} y^{\frac{k}{2}+\frac{1}{4}} f(z) \overline{\theta(z)}
      \overline{P_h^\kappa(z,\overline{s})}\, \frac{dxdy}{y^2} \\
  &= \int_0^\infty \int_0^1 y^{s+\frac{k}{2}-\frac{3}{4}} f(z) \overline{\theta(z)} \overline{e(hz)}
      \frac{dxdy}{y} \\
  &= \int_0^\infty \int_0^1 y^{s+\frac{k}{2}-\frac{3}{4}}
      \Big( \sum_{m_1 \geq 1} a(m_1) e^{2\pi i m_1 x - 2\pi  m_1 y} \Big) \\
  & \qquad \times \Big( \sum_{m_2 \geq 0} r_1(m_2) e^{-2\pi i m_2 x - 2\pi m_2 y} \Big)
      e^{-2\pi i h x -2\pi h y} \frac{dxdy}{y}.
\end{align*}%
The $x$-integral extracts those terms with $m_1-m_2 -h=0$, and the remaining $y$-integral evaluates in terms of the gamma function:
\begin{align*}
\langle y^{\frac{k}{2}+\frac{1}{4}} f \overline{\theta}, P_h^\kappa(\cdot, \overline{s}) \rangle
&= \sum_{m \geq 0} r_1(m) a(m+h) \int_0^\infty y^{s+\frac{k}{2}-\frac{3}{4}} e^{-4\pi (m+h) y} \frac{dy}{y} \\
&= \frac{\Gamma(s+\frac{k}{2}-\frac{3}{4})}{(4\pi)^{s+\frac{k}{2}-\frac{3}{4}}}
\sum_{m \geq 0} \frac{r_1(m) a(m+h)}{(m+h)^{s+\frac{k}{2}-\frac{3}{4}}}.
\end{align*}
Standard estimates show that this Dirichlet series converges absolutely for $\Re(s) > \frac{3}{4}$.
By rearranging, we obtain the following identity for the Dirichlet series $D_h(s)$ introduced in~\eqref{eq:D_h-definition}.

\begin{proposition} \label{prop:D_h-inner-product}
Fix $h>0$ and any $f(z) = \sum a(n) e(nz) \in S_k(\Gamma_0(N),\chi)$. We have
\[D_h(s) := \sum_{m \geq 0} \frac{r_1(m) a(m+h)}{(m+h)^{s+\frac{k}{2}-\frac{3}{4}}}
= \frac{(4\pi)^{s+\frac{k}{2}-\frac{3}{4}}
    \langle y^{\frac{k}{2}+\frac{1}{4}} f \overline{\theta}, P_h^\kappa(\cdot, \overline{s}) \rangle}
  {\Gamma(s+\frac{k}{2}-\frac{3}{4})}\]
in the region $\Re s > \frac{3}{4}$.
\end{proposition}

\section{Spectral Expansion of the Poincar\'e Series} \label{sec:spectral-expansion}

The series $D_h(s)$ has a meromorphic continuation to all $s \in \mathbb{C}$ obtained through spectral
expansion of the Poincar\'e series $P_h^\kappa(z,s)$ in Proposition~\ref{prop:D_h-inner-product}.
(See~\cite[\S15.3]{Stromberg08} for a good general reference on spectral expansions for general weight and the
shapes of each component.)
As a weight $\kappa = k -\frac{1}{2}$ object, this spectral expansion takes the form
\begin{align}
\begin{split} \label{eq:P_h-spectral-expansion}
  P_h^\kappa(z,s)
  &=  \sum_j \langle  P_h^\kappa(\cdot , s), \mu_j \rangle \mu_j(z)
    + \sum_{\ell} \langle P_h^\kappa(\cdot, s), R_\ell\rangle R_\ell(z) \\
  &\qquad + \frac{1}{4\pi i} \sum_{\mathfrak{a}} \int_{(\frac{1}{2})}
      \!\big\langle P_h^\kappa(\cdot, s), E_\mathfrak{a}^\kappa(\cdot,w;\chi)\big\rangle
      E_\mathfrak{a}^\kappa(z,w;\chi)\,dw,
\end{split}
\end{align}
in which $\{\mu_j\}$ denotes an orthonormal basis of weight $\kappa$ Maass
cuspforms of level $N$ and multiplier system $\chi \chi_{-1}^k \upsilon_\theta^{-1}$ which are eigenfunctions of the Hecke operators coprime to $N$, $\{R_\ell\}$ is a finite orthonormal basis of the residual spectrum of weight $\kappa$ and
multiplier $\chi \chi_{-1}^k \upsilon_\theta^{-1}$, $\mathfrak{a}$ ranges over the cusps of $\Gamma_0(N)$ which are singular with respect to this multiplier, and $E_\mathfrak{a}^\kappa(z,w;\chi)$ is
the weight $\kappa$ Eisenstein series with character $\chi \chi_{-1}^k$.  We refer to
the expressions at right in~\eqref{eq:P_h-spectral-expansion} as the discrete,
residual, and continuous spectra, respectively.

Inserting this spectral expansion into $D_h(s)$ as presented in Proposition~\ref{prop:D_h-inner-product}
gives a spectral expansion of the form
\begin{equation*}
  D_h(s) = \Sigma_{\mathrm{disc}}(s) + \Sigma_{\mathrm{res}}(s) + \Sigma_{\mathrm{cont}}(s),
\end{equation*}
which we now describe more fully.

\subsection{The Discrete Spectrum} \label{subsec:discrete-spectrum}

The Maass cuspforms in the discrete spectrum have Fourier expansions of the form
\begin{align} \label{eq:Fourier-Maass}
\mu_j(z) &= \sum_{n \neq 0} \rho_j(n) W_{\frac{n \kappa}{2 \vert n \vert}, it_j}(4\pi \vert n \vert y) e(nx),
\end{align}
in which $W_{\eta,\nu}(z)$ is the $\mathrm{GL}(2)$ Whittaker function.
By unfolding the Poincar\'e series $P_h^\kappa(z,s)$ and applying the integral
formula~\cite[7.621(3)]{GradshteynRyzhik07} for the resulting $y$-integral, we evaluate $\langle
P_h^\kappa(\cdot, s), \mu_j \rangle$ and conclude that the discrete spectrum's contribution towards
$D_h(s)$ equals
\begin{align} \label{eq:discrete-spectrum}
\hspace{-5 mm} \Sigma_{\mathrm{disc}} :=
\frac{(4\pi)^{\frac{k}{2}+\frac{1}{4}}}{h^{s-1}}
\! \sum_j \! \frac{\Gamma(s-\frac{1}{2}+it_j)\Gamma(s-\frac{1}{2}-it_j)}
       {\Gamma(s-\frac{k}{2}+\frac{1}{4})\Gamma(s+\frac{k}{2}-\frac{3}{4})}
   \rho_{j}(h) \langle y^{\frac{k}{2}+\frac{1}{4}} f \overline{\theta}, \mu_j \rangle.
\end{align}

We will establish in~\S{\ref{subsec:discrete-spectrum-growth}} that this series for $\Sigma_{\mathrm{disc}}$ converges everywhere away from poles.
Assuming this, $\Sigma_{\mathrm{disc}}(s)$ defines a meromorphic function which is analytic in $\Re s >
\frac{1}{2}+\sup_j \lvert \Im t_j \rvert$.

The orthonormal basis of Maass forms $\{\mu_j\}$ includes a finite subset of
distinguished forms arising from lifts of holomorphic cuspforms of weight
$\ell$, with $0 < \ell \leq \kappa$ and $\ell \equiv \kappa \bmod 2$.
(See for example \S3.10 of~\cite{GoldfeldHundley11}.)
These Maass forms have spectral types $\pm it_j = \frac{\ell-1}{2}$ and their
contribution towards $\Sigma_{\mathrm{disc}}$ may be written

\begin{align} \label{eq:Sigma_hol-definition}
\Sigma_{\mathrm{hol}}(s) :=
\frac{(4\pi)^{\frac{k}{2}+\frac{1}{4}}}{h^{s-1}}
\! \sum_{\substack{0 < \ell \leq \kappa \\ \ell \equiv \kappa(2) }}
	\sum_{\{g_{\ell j}\}} \! \frac{\Gamma(s-\frac{\ell}{2})\Gamma(s+\frac{\ell}{2}-1)}
       {\Gamma(s-\frac{\kappa}{2})\Gamma(s+\frac{\kappa-1}{2})}
   \rho_{\ell j}(h) \langle y^{\frac{k}{2}+\frac{1}{4}} f \overline{\theta}, g_{\ell j} \rangle,
\end{align}
in which $\{g_{\ell j}\}_j$ denotes an orthonormal basis of Maass lifts from holomorphic forms of weight $\ell$ to Maass forms of weight $\kappa$, with Fourier coefficients $\rho_{\ell j}(n)$. We refer to $\Sigma_{\mathrm{hol}}(s)$ as the contribution of the \emph{(Maass lifted) holomorphic spectrum}.

Note that the gamma ratio $\Gamma(s-\frac{\ell}{2})/\Gamma(s-\frac{\kappa}{2})$
contributes no poles and that the gamma factor $\Gamma(s+\frac{\ell}{2}-1)$ is
analytic in $\Re s > 1 - \frac{\ell}{2}$. Thus $\Sigma_{\mathrm{hol}}(s)$ has a
potential simple pole at $s= \frac{3}{4}$ (from $\ell = \frac{1}{2}$ when $k$
is odd) and is otherwise analytic in $\Re s > \frac{1}{4}$.

By the Shimura correspondence for Maass forms (see~\cite{KatokSarnak93}), the other summands in
$\Sigma_{\mathrm{disc}}(s)$ are analytic in $\Re s > \frac{1}{2} + \frac{\Theta}{2}$,
where $\Theta$ denotes the progress towards the Selberg eigenvalue conjecture
as before. Thus $\Sigma_{\mathrm{disc}}(s)$ is analytic in $\Re s > \frac{1}{2}
+ \frac{\Theta}{2}$, except for a potential simple pole at $s= \frac{3}{4}$
when $k$ is odd. We now examine this pole further.

The pole at $s = \frac{3}{4}$ in $\Sigma_{\mathrm{disc}}(s)$, if it occurs, is
localized to the terms in $\Sigma_{\mathrm{hol}}$ coming from Maass lifts of
holomorphic cuspforms of weight $\frac{1}{2}$.
By~\cite[\S{2}]{SerreStark77}, the space of weight $\frac{1}{2}$ modular forms on $\Gamma_0(N)$ with character $\chi$ has a basis of theta functions of the form
\[
	\theta_{\psi,t}(z) = \sum_{n \in \mathbb{Z}} \psi(n) e(tn^2 z),
\]
where $\psi$ is an even primitive character of conductor $L$ with $4 L^2 t \mid N$ and $\chi(n) = \psi(n) (\frac{t}{n})$ for $(n,N)=1$. Moreover, the subspace of weight $\frac{1}{2}$ cuspforms on $\Gamma_0(N)$ with character $\chi$ is spanned by those $\theta_{\psi,t}$ for which $\psi$ is not \emph{totally even}, i.e.\ $\psi$ is not the square of another character.

In particular, this cuspidal space is empty whenever $N/4$ is square-free,
since this condition forces $L=1$, so $\psi$ is trivial and therefore totally even.
Thus $\Sigma_{\mathrm{hol}}(s)$ is analytic in $\Re s > \frac{1}{4}$ whenever
$N/4$ is square-free.

\begin{remark} \label{rem:holomorphic-main-term}
The potential pole of $\Sigma_{\mathrm{hol}}(s)$ at $s= \frac{3}{4}$ vanishes in many cases. If this pole persists, then
\begin{itemize}
	\item[a.] $k$ must be odd;
	\item[b.] $N/4$ must admit a square divisor. More precisely, \cite[Corollary~2]{SerreStark77} implies that $N$ must be divisible by $64p^2$ or $4 p^2 q^2$, where $p$ and $q$ are distinct odd primes;
	\item[c.] the square-free part of $h$ must divide $N$, since $\rho_{\ell j}(h) =0$ for the relevant Maass lifts otherwise.
\end{itemize}
In other cases, this pole can occur. We use Poincar\'e series to provide a class of examples in Remark~\ref{rem:Maass-lift-poincare-example}. For future reference, we define
\begin{align} \label{eq:discrete-spectrum-residue}
	b_{f,h} := \Res_{s = \frac{3}{4}} \Sigma_{\mathrm{hol}}(s).
\end{align}
This residue can be written as a sum indexed by an orthonormal basis of weight $\kappa$ Maass forms lifted from holomorphic forms of weight $\frac{1}{2}$.
\end{remark}

\subsubsection{Extended Remark}\label{extended_remark}

We note that the potential pole at $s=\frac{3}{4}$ in $\Sigma_{\mathrm{hol}}(s)$ was omitted from consideration
in~\cite{TemplierTsimerman12}, as that argument neglected the
polar contribution from the Maass lifted holomorphic spectrum.

We also note that we aren't the first to notice this pole. Watkins observes the same potential pole following equation (8) in~\cite[\S{5.5}]{Watkins2019} (and notes in his footnote~15 that this pole does not occur in the case $\kappa \equiv \frac{3}{2} \bmod 2$).
In \S5.5.2, Watkins enumerates several other small omissions in \S4
of~\cite{TemplierTsimerman12}.

Our treatment of the holomorphic spectrum $\Sigma_{\mathrm{hol}}$ resembles the method used in~\cite[\S{4.7}]{TemplierTsimerman12} to address the residual spectrum $\Sigma_{\mathrm{res}}$. We discuss Templier--Tsimerman's treatment of $\Sigma_{\mathrm{res}}$ in~\S{\ref{subsec:residual-spectrum}}. In particular, we conjecture that the potential pole of $\Sigma_{\mathrm{hol}}(s)$ at $s=\frac{3}{4}$ vanishes under the same conditions which force a potential pole in $\Sigma_{\mathrm{res}}(s)$ to vanish, as described in Remark~\ref{rem:general-main-term}. That is, in addition to the conditions listed in Remark~\ref{rem:holomorphic-main-term}, we conjecture that the potential pole of $\Sigma_{\mathrm{hol}}(s)$ at $s=\frac{3}{4}$ vanishes when $f(z)$ is not dihedral. We do not attempt to prove this conjecture.

Aside from contributing a potential pole at $s = \frac{3}{4}$, the holomorphic
spectrum $\Sigma_{\mathrm{hol}}(s)$ can be bounded in essentially the exact
same way and with the same bounds as the rest of $\Sigma_{\mathrm{disc}}(s)$.
In particular, we observe that the technical bounds we give in \S5 and \S6 for
the discrete spectrum are more constrained by Maass forms that are not lifts of
holomorphic forms.

\subsection{The Continuous Spectrum}

The Eisenstein series in the continuous spectrum have Fourier expansions of the form
\begin{align} \label{eq:Fourier-Eisenstein}
E_\mathfrak{a}^\kappa(z,w;\chi) &= \delta_{[\mathfrak{a}=\infty]} y^w + \rho_{\mathfrak{a}}(0,w) y^{1-w}
  \\
  &\qquad \qquad \quad\,  + \sum_{n \neq 0} \rho_{\mathfrak{a}}(n,w) W_{ \! \frac{n\kappa}{2 \vert n \vert}, w-\frac{1}{2}}(4\pi \vert n \vert y) e(nx),
\end{align}
in which $\delta_{[\cdot]}$ denotes the Kronecker delta. By unfolding $P_h^\kappa(z,s)$ as before, we determine that the continuous spectrum's contribution towards $D_h(s)$ may be written
\begin{align} \label{eq:continuous-spectrum}
 \Sigma_{\mathrm{cont}}(s) := \frac{(4\pi)^{\frac{k}{2}-\frac{3}{4}}}{h^{s-1}} \sum_{\mathfrak{a}}
	&\int_{-\infty}^\infty \frac{\Gamma(s-\frac{1}{2} + it)\Gamma(s-\frac{1}{2}-it)}{\Gamma(s-\frac{k}{2}+\frac{1}{4})\Gamma(s+\frac{k}{2}-\frac{3}{4})} \\ \nonumber
& \quad \times \rho_\mathfrak{a}(h,\tfrac{1}{2}+it) \langle y^{\frac{k}{2}+\frac{1}{4}} f \overline{\theta} , E_\mathfrak{a}^\kappa(\cdot ,\tfrac{1}{2}+it;\chi) \rangle \, dt.
\end{align}
We will prove in~\S{\ref{subsec:continuous-spectrum-growth}} that
$\Sigma_{\mathrm{cont}}$ converges everywhere away from poles and therefore
defines an analytic function of $s$ in $\Re s > \frac{1}{2}$.
Through delicate contour shifting (as in~\cite[\S{4}]{HoffsteinHulse13} or~\cite[\S{3.3.2}]{HKLDWSphere}), one may show that
$\Sigma_{\mathrm{cont}}$ extends meromorphically to all $\Re s \in \mathbb{C}$, though we only consider $\Sigma_{\mathrm{cont}}$ in $\Re s >
\frac{1}{2}$ in this work.

\subsection{The Residual Spectrum} \label{subsec:residual-spectrum}

We conclude this section by discussing the residual spectrum, which contributes a term of the form
\[
\Sigma_{\mathrm{res}}(s)
  := \frac{(4\pi)^{s+\frac{k}{2}-\frac{3}{4}}}{\Gamma(s+\frac{k}{2}-\frac{3}{4})}
    \sum_\ell \langle R_\ell, P_h^\kappa(\cdot, \overline{s}) \rangle
    \langle y^{\frac{k}{2}+\frac{1}{4}} f \overline{\theta}, R_\ell \rangle
\]
towards $D_h(s)$. The material here follows~\cite[\S{3.5} and \S{4.7}]{TemplierTsimerman12}, with minor modifications and elaborations.

By the theory of raising operators, the weight $\kappa = k-\frac{1}{2}$ residual spectrum lifts from weight $\frac{1}{2}$ forms if $k$ is odd and from weight $\frac{3}{2}$ forms is $k$ is even. By~\cite[\S{2}]{Duke88}, the residual spectrum does not appear in weight $\frac{3}{2}$, so we restrict to $k$ odd for the rest of this section.

Suppose initially that $N/4$ is square-free and odd. By~\cite{SerreStark77}, the weight $\frac{1}{2}$ residual spectrum
appears only when $\chi = (\frac{\cdot}{N/4})$. In this case, the space is one-dimensional and spanned by the theta function
\[
y^{\frac{1}{4}} \theta_N(z) = y^{\frac{1}{4}} \sum_{n \in \mathbb{Z}} e(\tfrac{N n^2 z}{4})
= y^{\frac{1}{4}}
  + \sum_{n \neq 0} (\pi N n^2)^{-\frac{1}{4}} W_{\frac{1}{4},\frac{1}{4}}(\pi N n^2 y) e(\tfrac{N n^2 x}{4}).
\]
Under these assumptions, we have
\[
  \Sigma_{\mathrm{res}}(s) :=
    \frac{(4\pi)^{s+\frac{k}{2}-\frac{3}{4}}}{\Gamma(s+\frac{k}{2}-\frac{3}{4})}
    \langle u, P_h^\kappa(\cdot, \overline{s}) \rangle
    \langle y^{\frac{k}{2}+\frac{1}{4}} f \overline{\theta}, u\rangle,
\]
in which $u(z)$ denotes the $L^2$-normalized lift of $y^{\frac{1}{4}}\theta_N(z)$ to weight $\kappa$ under the raising operators. As the weight $\eta$ raising operator $R_\eta = i y \frac{\partial}{\partial x} + y \frac{\partial}{\partial y} + \frac{\eta}{2}$ maps $y^{1/4}$ to $(\frac{\eta}{2}+\frac{1}{4})y^{1/4}$ and $W_{\frac{\eta}{2},\frac{1}{4}}(4\pi m y) e(mx)$ to $-W_{\frac{\eta+2}{2},\frac{1}{4}}(4\pi m y) e(mx)$, the (non-normalized) lift of $y^{\frac{1}{4}} \theta_N(z)$ to weight $\kappa$ equals
\[
U(z):= y^{\frac{1}{4}} \!\!\!
  \prod_{\substack{\frac{1}{2} \leq \eta < \kappa \\ \eta \equiv \frac{1}{2} \bmod 2}}
      \!\!\! (\tfrac{\eta}{2}+\tfrac{1}{4})
  + (-1)^{\frac{k-1}{2}} \sum_{n \neq 0} (\pi N n^2)^{-\frac{1}{4}} W_{\frac{\kappa}{2},\frac{1}{4}}(\pi N n^2 y) e(\tfrac{N n^2 x}{4}).
\]

To relate $U(z)$ and $u(z)$, we consider the effect of the raising operators on $L^2$ norms. Let $\mu$ denote any Maass cusp form of weight $\eta$ and type $\nu$ on $\Gamma_0(N)$ and let $L_\eta = -iy \frac{\partial}{\partial x} + y \frac{\partial}{\partial y} - \frac{\eta}{2}$ denote the weight $\eta$ Maass lowering operator. By combining~\cite[(3.9.4)]{GoldfeldHundley11} with the last offset equation on~\cite[p.~90]{GoldfeldHundley11}, we produce
\begin{align} \label{eq:raising-L^2}
\Vert R_\eta \mu \Vert^2
  & = \langle R_\eta \mu, R_\eta \mu \rangle
  = \langle \mu, - L_{\eta+2} R_\eta \mu \rangle  \\
  &= \langle \mu, \Delta_\eta \mu \rangle + \tfrac{\eta}{2}(1+\tfrac{\eta}{2}) \langle \mu, \mu \rangle
  = \big(\tfrac{1}{4}+ \nu^2 + \tfrac{\eta}{2}(1+\tfrac{\eta}{2}) \big) \Vert \mu \Vert^2.
\end{align}
We note that this is analogous to~\cite[equation (3.20)]{TemplierTsimerman12}.

In the special case $\nu = i/4$, we conclude that
\[\Vert U \Vert^2 = \Vert y^{\frac{1}{4}} \theta_N \Vert^2
  \!\! \prod_{\substack{ \frac{1}{2} \leq \eta < \kappa \\ \eta \equiv \frac{1}{2} \bmod 2}} \!\!
  (\tfrac{\eta}{2}+\tfrac{1}{4})(\tfrac{\eta}{2}+\tfrac{3}{4})
  = \Vert y^{\frac{1}{4}} \theta_N \Vert^2 \prod_{j =1}^{\frac{k-1}{2}} j(j-\tfrac{1}{2}),
\]
which implies that
\begin{align*}
  u(z) &= \frac{y^{\frac{1}{4}}}{\Vert y^{\frac{1}{4}} \theta_N \Vert} \!\!
  \prod_{\substack{\frac{1}{2} \leq \eta < \kappa \\ \eta \equiv \frac{1}{2} \bmod 2}} \!\!
  \frac{(\tfrac{\eta}{2} + \tfrac{1}{4})^{\frac{1}{2}}}{(\tfrac{\eta}{2} + \tfrac{3}{4})^{\frac{1}{2}}}
  +
  \frac{(-1)^{\frac{k-1}{2}} d_k^{-\frac{1}{2}}}{\Vert y^{\frac{1}{4}} \theta_N \Vert}
  \sum_{n \neq 0} \frac{%
    W_{\frac{\kappa}{2},\frac{1}{4}}(\pi N n^2 y)
  }{%
    (\pi N n^2)^{\frac{1}{4}}
  } e(\tfrac{N n^2 x}{4}),
\end{align*}
in which $d_k = \prod_{j=1}^{\frac{k-1}{2}} j(j-\frac{1}{2})$, cf.~\cite[\S{4.7}]{TemplierTsimerman12}.

We now compute the two inner products which appear in $\Sigma_{\mathrm{res}}$.
To begin, we unfold the Poincar\'e series and apply~\cite[7.621(3)]{GradshteynRyzhik07} to produce
\begin{align*}
 \langle u, P_h^\kappa(\cdot, \overline{s}) \rangle
&  = \frac{2 (-1)^{\frac{k-1}{2}} d_k^{-\frac{1}{2}}  \delta_{[N n^2 =4h]}}{(4\pi h)^{\frac{1}{4}} \Vert y^{\frac{1}{4}} \theta_N \Vert} \int_0^\infty y^{s-1} W_{\frac{\kappa}{2},\frac{1}{4}}(4\pi h y) e^{-2\pi h y} \frac{dy}{y} \\
& = \frac{(-1)^{\frac{k-1}{2}} d_k^{-\frac{1}{2}} r_1(\frac{4h}{N}) \Gamma(s-\frac{1}{4})\Gamma(s-\frac{3}{4})}{(4\pi h)^{s-\frac{3}{4}} \Vert y^{\frac{1}{4}} \theta_N \Vert \Gamma(s-\frac{\kappa}{2})}.
\end{align*}

To compute $\langle y^{\frac{k}{2}+\frac{1}{4}} f \overline{\theta}, u \rangle$, we recognize $u(z)$ as a multiple of the residue at $w=\frac{3}{4}$ of the Eisenstein series $E_\infty^\kappa(z,w)$ of level $N$, weight $\kappa$, and trivial character (i.e.~\eqref{eq:P_h-definition}, with $h=0$ and $\chi$ trivial).
We have $\Res_{w=3/4} E_\infty^\kappa(z,w) = u(z) d_k^{1/2} \Vert y^{\frac{1}{4}} \theta_N \Vert^{-1}$ from~\cite[\S{4.7}]{TemplierTsimerman12}, and hence
\begin{align*}
\langle y^{\frac{k}{2}+\frac{1}{4}} f \overline{\theta}, u \rangle
  &= d_k^{-\frac{1}{2}}  \Vert y^{\frac{1}{4}} \theta_N \Vert \cdot
  \Res_{w=\frac{3}{4}}
    \langle y^{\frac{k}{2}+\frac{1}{4}} f \overline{\theta}, E_\infty^\kappa(z,\overline{w}) \rangle \\
&= d_k^{-\frac{1}{2}}  \Vert y^{\frac{1}{4}} \theta_N \Vert
    \cdot \Res_{w=\frac{3}{4}} \sum_{n=1}^\infty 2a(n^2)
    \int_0^\infty y^{w+\frac{k}{2}-\frac{3}{4}} e^{-2\pi n^2 y} \frac{dy}{y} \\
&= 2d_k^{-\frac{1}{2}}  \Vert y^{\frac{1}{4}} \theta_N \Vert
    \cdot \Res_{w=\frac{3}{4}} \frac{\Gamma(w+\frac{k}{2}-\frac{3}{4})}{(2\pi)^{w+\frac{k}{2}-\frac{3}{4}}}
    \sum_{n=1}^\infty \frac{a(n^2)}{n^{2w+k-\frac{3}{2}}}.
\end{align*}
The Dirichlet series in the line above equals $L(2w-\frac{1}{2}, \mathrm{Sym}^2 f)/\zeta(4w-1)$, in which $L(s, \mathrm{Sym}^2 f)$ is the symmetric square $L$-function of $f(z)$. Putting everything together, we derive the explicit formula
\begin{align} \label{eq:residual-squarefree-level}
\Sigma_{\mathrm{res}}(s) &=
\frac{2^{\frac{k}{2}}\Gamma(\frac{k}{2}) (-1)^{\frac{k-1}{2}} r_1(\frac{4h}{N})\Gamma(s-\frac{1}{4})\Gamma(s-\frac{3}{4})}
    {\zeta(2) d_k h^{s-\frac{3}{4}} \Gamma(s+\frac{k}{2}-\frac{3}{4})\Gamma(s-\frac{\kappa}{2})}
\cdot \Res_{w=1} L(w,\mathrm{Sym}^2 f) \\ \nonumber
& =
\frac{2^{\frac{k}{2}} \sqrt{\pi} (-1)^{\frac{k-1}{2}} r_1(\frac{4h}{N})}{\zeta(2) \Gamma(\frac{k+1}{2}) h^{s-\frac{3}{4}}} \cdot
\frac{\Gamma(s-\frac{1}{4})\Gamma(s-\frac{3}{4})}{\Gamma(s+\frac{k}{2}-\frac{3}{4})\Gamma(s-\frac{\kappa}{2})}
\cdot \Res_{w=1} L(w,\mathrm{Sym}^2 f),
\end{align}
in which the second line follows from the identity $d_k = \Gamma(k)/2^{k-1}$ and the gamma
duplication formula. In conclusion, we note that $\Sigma_{\mathrm{res}}(s)$ is meromorphic in
$\mathbb{C}$ and analytic in $\Re s > \frac{3}{4}$, with a simple pole at $s=\frac{3}{4}$ only when $k$ is odd,
$f(z)$ is dihedral, $\chi = (\frac{\cdot}{N/4})$, and $r_1(4h/N) \neq 0$.

\begin{remark} \label{rem:general-main-term}
For $N$ with generic square part and odd level $k$, the residual spectrum has a basis consisting of (finitely many) lifts of theta functions of the form $y^{1/4} \theta_{\psi,t}(z)$ introduced in~\S{\ref{subsec:discrete-spectrum}}.
Since these lifts can all be recognized as residues of weight $\kappa$ Eisenstein series on $\Gamma_0(N)$, Templier--Tsimerman claim that analogous expressions for $\Sigma_{\mathrm{res}}$ exist for general $N$. In particular,
\begin{align} \label{eq:residual-general}
\Sigma_{\mathrm{res}}(s)
  = c_{f,h} \cdot \frac{(-1)^{\frac{k-1}{2}} \pi^{\frac{1}{2}} }{h^{s-\frac{3}{4}}}
      \cdot \frac{\Gamma(s-\frac{1}{4})\Gamma(s-\frac{3}{4})}
          {\Gamma(s+\frac{k}{2}-\frac{3}{4})\Gamma(s-\frac{\kappa}{2})},
\end{align}
in which $c_{f,h}=0$ unless the following conditions from~\cite[\S{4.7}]{TemplierTsimerman12} are met:
\begin{itemize}
\item[a.] $f(z)$ is a dihedral form of odd weight $k$;
\item[b.] $\frac{N}{4} \mid h$ with $h>0$ (correcting an error in~\cite{TemplierTsimerman12});
\item[c.] If $h$ has square-free part $h_0$ and $L$ is the conductor of $\chi \chi_h$, then $h_0 L^2$ divides $N/4$.
\end{itemize}
\end{remark}
\noindent The non-obvious choice of normalization in~\eqref{eq:residual-general} has been chosen so that the residue of $\Sigma_{\mathrm{res}}(s)$ at $s=\frac{3}{4}$ equals $c_{f,h}$.

We remark that the generalized argument of Templier--Tsimerman is incomplete, as it ignores complications with the slash operators and the contribution of oldforms, which unfold at levels lower than $N$. A similar oversight regarding mismatched levels of unfolding in~\cite{TemplierTsimerman12} is noted by~\cite{Watkins2019} prior to footnote 17.

\section{Averages for Fourier Coefficients} \label{sec:fourier-coefficient-averages}

To study the partial sums $\sum_{n \leq X} A(n^2+h)$, we require certain information about the growth of $D_h(s)$ in vertical strips.  To this end, we present here a few point-wise and on-average bounds involving the Fourier coefficients of half-integral weight Maass forms.

To give context for our on-average results, we first quote a conditional point-wise result for the Fourier coefficients of a Maass form.

\begin{lemma}[cf.\ \cite{Duke88}, Theorem 5] \label{lem:Duke-bound}
Assume the generalized Lindel\"{o}f hypothesis and the Ramanujan--Petersson conjecture.
Let $\mu_j$ be an $L^2$-normalized Maass form of (half-integral) weight $\kappa$ on $\Gamma_0(N)$, with Fourier expansion of the form~\eqref{eq:Fourier-Maass}. If $n$ is a fundamental discriminant, we have
\[\rho_j(n) \ll_{\kappa,N,\epsilon} (1+\vert t_j \vert)^{- \frac{\kappa}{2} \mathrm{sgn}(n) +\epsilon} e^{\frac{\pi}{2} \vert t_j \vert} \vert n \vert^{-\frac{1}{2}+\epsilon}\]
for any $\epsilon > 0$.
\end{lemma}

The main content of~\cite[Theorem 5]{Duke88} is an unconditional version of Lemma~\ref{lem:Duke-bound}, though
the $t_j$-dependence in the unconditional bound is too weak for our applications. For adequate unconditional
results, we require some amount of averaging. Our first average is a spectral second moment for half-integral
weight Maass forms which improves~\cite[Lemma 5]{Blomer08}.

\begin{proposition}\label{prop:spectral-average-appendix}
	Let $\{\mu_j\}$ denote an orthonormal basis of weight $\kappa =k-\frac{1}{2}$ Maass forms on $\Gamma_0(N)$ with multiplier system $\chi \chi_{-1}^k \upsilon_\theta^{-1}$ and Fourier expansions of the form~\eqref{eq:Fourier-Maass}.
	For any $T \geq 1$ and $\epsilon > 0$,
	\[
	\sum_{\substack{\vert t_j \vert \leq T \\ t_j \in \mathbb{R}}}
	\frac{|n|}{\cosh(\pi t_j )} \vert \rho_{j}(n) \vert^2 (1+ \vert t_j \vert)^{\kappa \, \mathrm{sgn}(n)}
	\ll_{k,N,\epsilon} T^2 + \vert n \vert^{\frac{1}{2}+\epsilon}.
	\]
\end{proposition}

We note that Proposition~\ref{prop:spectral-average-appendix} implies that Duke's conditional result holds
unconditionally in the long average over $\vert t_j \vert \sim T$ when $\vert n \vert  \ll T^4$. The proof of Proposition~\ref{prop:spectral-average-appendix} is due to Raphael Steiner and appears in Appendix~\ref{sec:appendix}.

We can also produce strong on-average results for $\rho_j(m)$ in the $m$-aspect
by refining the method leading to~\cite[Lemma 19.3]{DFI02}. We show how to obtain the following
$m$-average upper bound.

\begin{proposition} \label{prop:m-average}
Let $\mu_j$ be an $L^2$-normalized weight $\kappa = k-\frac{1}{2}$ Maass form on $\Gamma_0(N)$ with multiplier $\chi \chi_{-1}^k \upsilon_\theta^{-1}$ and Fourier expansion~\eqref{eq:Fourier-Maass}. Then
\[\sum_{m \sim M} \vert \rho_{j}(\pm m) \vert^2 \ll_{\kappa,N} (1+\vert t_j \vert)^{ \mp \kappa} \Big(1+\frac{\vert t_j \vert}{M}\Big) e^{\pi \vert t_j \vert}\]
for any $M > 1$ and any choice of sign $\pm$.
\end{proposition}

This result implies that Duke's conditional result holds unconditionally in the long average over $\vert m \vert \sim M$, provided $\vert t_j \vert \ll M$. To prove Proposition~\ref{prop:m-average}, we require a strengthened form of~\cite[Lemma 19.2]{DFI02}, which in turn relies on the following uniform estimate for the Whittaker function.

\begin{lemma} \label{lem:uniform-Whittaker-bound}
For $y>0$, $\eta$ real, and $t \geq 1$, we have
\[W_{\eta,it}(y) \ll_\eta t^{\eta - \frac{1}{2}} e^{-\frac{\pi}{2} t} \cdot y^{\frac{1}{2}}\]
uniformly in the interval $0 \leq y \leq \frac{3}{2} t$.
\end{lemma}

\begin{proof}
This result essentially follows from the Liouville--Green approximation of the differential equation for $W_{\eta,it}(y)$, as described in~\cite[ch. 6]{Olver74}.  Our specific application to the Whittaker function is not new, and indeed appears in~\cite{Dunster03}.

In particular, Lemma~\ref{lem:uniform-Whittaker-bound} follows from~\cite[p.\ 210-211]{Dunster03} by removing $\eta$-uniformity from the ``$s \leq s^+$'' case. (Note that $\mu^{-1/6}$ in~\cite[(4.14)]{Dunster03} should read $\mu^{1/6}$.) The restriction to $y \leq \frac{3}{2} t$ (as opposed to $y \approx 2 t$, where the Whittaker function stops oscillating) simplifies our expression further by avoiding complications near the ``turning point'' at $s^+$. (Away from the turning point, we note that one may appeal to the simpler error analysis of~\cite[ch.~6, \S{1-5}]{Olver74} and avoid Dunster's use of~\cite[ch.~11]{Olver74} entirely.)
\end{proof}

We now give a strengthened form of~\cite[Lemma 19.2]{DFI02}.

\begin{lemma} \label{lem:strong-DFI-technical-lemma}
There exists a constant $\alpha > 0$ depending only on $\kappa$ for which
\[\int_{\alpha t}^{\infty} W_{\pm \frac{\kappa}{2},it}(4\pi y)^2 \frac{dy}{y^2} \gg_\kappa  t^{\pm \kappa-1} e^{-\pi t},\]
uniformly in $t \geq 1$.
\end{lemma}

\begin{proof}
Let $\eta = \pm \frac{\kappa}{2}$. From~\cite[7.611(4)]{GradshteynRyzhik07} we derive
\begin{align*}
& \int_0^\infty W_{\eta, it}(4\pi y)^2 \frac{dy}{y}
= \frac{\pi}{\sin(2\pi i t)} \cdot \frac{\psi(\frac{1}{2}-\eta+it)-\psi(\frac{1}{2}-\eta-it)}{\Gamma(\frac{1}{2}-\eta+it)\Gamma(\frac{1}{2}-\eta-it)},
\end{align*}
in which $\psi(z)$ is the digamma function. We apply Stirling's approximation and the asymptotic $\psi(\frac{1}{2}-\eta+it)-\psi(\frac{1}{2}-\eta-it) = i \pi + O_\eta(1/t)$ to produce
\[\int_0^\infty W_{\eta, it}(4\pi y)^2 \frac{dy}{y} = \pi t^{2\eta} e^{-\pi t}(1+ O_\eta(t^{-1})).\]

To control the behavior of $W_{\eta,it}(4\pi y)$ near $y=0$, we apply Lemma~\ref{lem:uniform-Whittaker-bound} and integrate to obtain
\[
  \int_0^{\alpha t} W_{\eta,it}(4\pi y)^2 \frac{dy}{y} \ll_\eta \alpha t^{2\eta} e^{-\pi t},
\]
uniformly in $t \geq 1$ and $\alpha \leq 3/(8\pi)$. In particular, there exists a small constant $\alpha$ depending only on $\eta$ for which
\begin{align} \label{eq:whittaker-integral-single-y-bound}
\int_{\alpha t}^\infty W_{\eta, it}(4\pi y)^2 \frac{dy}{y} \gg_\eta t^{2\eta} e^{-\pi t}.
\end{align}

As in the proof of~\cite[Lemma 19.2]{DFI02}, integration by parts implies the existence of some $\beta$ depending only on $\eta$ for which~\eqref{eq:whittaker-integral-single-y-bound} holds when the domain of integration is restricted further to $\alpha t \leq y \leq \beta t$. Then
\begin{align*}
& \int_{\alpha t}^{\infty}  W_{\eta,it}(4\pi y)^2 \frac{dy}{y^2}
  \geq  \int_{\alpha t}^{\beta t}  W_{\eta,it}(4\pi y)^2 \frac{dy}{y^2} \\
& \qquad  \gg_\eta \vert t \vert^{-1} \int_{\alpha t}^{\beta t}
    W_{\eta,it}(4\pi y)^2 \frac{dy}{y}
 \gg_\eta t^{2\eta -1} e^{-\pi t},
\end{align*}
which completes the proof.
\end{proof}

We now return to the proof of Proposition~\ref{prop:m-average}.

\begin{proof}[Proof of Proposition~\ref{prop:m-average}]
Our proof adapts the proof of~\cite[Lemma 19.3]{DFI02}. Parseval's identity gives
\[
  \int_0^1 \vert \mu_j (z) \vert^2 dx
    = \sum_{n \neq 0} \vert \rho_j(n) \vert^2
        W_{\frac{\kappa n }{2 \vert n\vert}, it_j}(4\pi \vert n \vert y)^2.
\]
Since every orbit $\{\gamma z : z \in \Gamma_0(N)\}$ has $O_N(1+Y^{-1})$ points in $[0,1] \times (Y,\infty)$, integrating over $y \geq Y$ produces
\begin{align}
\begin{split} \label{eq:DFI-inequality}
1+\frac{1}{Y}
  & \gg_N \int_Y^\infty \int_0^1 \vert \mu_j (z) \vert^2 \frac{dx dy}{y^2} \\
&  = \sum_{n \neq 0} \vert \rho_j(n) \vert^2 \int_Y^\infty
    W_{\frac{\kappa  n }{2 \vert n \vert}, it_j}(4\pi \vert n \vert y)^2 \frac{dy}{y^2}.
\end{split}
\end{align}

Lemma~\ref{lem:strong-DFI-technical-lemma} and a change of variables implies that
\begin{align*}
  \int_{Y}^{\infty} W_{\frac{\kappa n}{2 \vert n \vert}, it_j}(4\pi \vert n \vert y)^2 \frac{dy}{y^2}
  \gg_\kappa \vert n \vert \cdot \vert t_j\vert^{\kappa \, \mathrm{sgn}(n)-1} e^{-\pi \vert t_j \vert},
\end{align*}
provided that $t_j$ is real with $\vert t_j  \vert \geq 1$ and $\vert n \vert Y \leq \alpha \vert t_j \vert$. We set $Y = \alpha \vert t_j \vert M^{-1}$ and deduce from~\eqref{eq:DFI-inequality} that
\[1+ \frac{M}{\alpha \vert t_j \vert}
	\gg_{\kappa,N} \sum_{\vert n \vert \leq M} \vert n \vert \cdot \vert \rho_{j}(n) \vert^2  \cdot \vert t_j \vert^{\kappa \, \mathrm{sgn}(n)-1} e^{-\pi \vert t_j \vert},
\]
which proves the proposition for large real $t_j$ after restricting to $\vert n \vert
\sim M$ and simplifying.

Otherwise, we assume that $t_j$ is real or purely imaginary, with $\vert t_j \vert \leq c_\kappa$ for some constant depending only on $\kappa$. By setting $Y = 1/(4\pi M)$ in~\eqref{eq:DFI-inequality} and changing variables $y \mapsto y/(4\pi \vert n \vert)$, it suffices to prove that
\begin{align} \label{eq:Whittaker-small-t_j-bound}
	\int_1^\infty W_{\frac{\kappa n}{2 \vert n \vert}, it_j}(y)^2 \frac{dy}{y^2} \gg_{\kappa} 1,
\end{align}
uniformly in real and purely imaginary $t_j$ with $\vert t_j \vert \leq c_\kappa$. To prove this, we note that left-hand side of~\eqref{eq:Whittaker-small-t_j-bound} is a continuous function of $t_j$ and thus attains a global minimum depending only on $\kappa$ (and $\mathrm{sgn}(n)$). Since the integrand is non-negative, this global minimum is non-negative since the Whittaker function is not identically zero.
\end{proof}

\section{An Average Involving Inner Products} \label{sec:D-estimate}

In addition to the coefficient bounds from section~\S{\ref{sec:fourier-coefficient-averages}}, we require estimates involving $\langle y^{\frac{k}{2}+\frac{1}{4}} f \overline{\theta}, \mu_j \rangle$;
specifically, we would like to estimate the sum
\begin{align} \label{eq:D-definition}
\mathfrak{D}:= \sum_{\vert t_j \vert \sim T} \vert \langle y^{\frac{k}{2}+\frac{1}{4}} f \overline{\theta}, \mu_j \rangle\vert^2 e^{\pi \vert t_j \vert}.
\end{align}

The main result in this section is the following bound for $\mathfrak{D}$.

\begin{theorem}\label{thm:D-estimate}
Fix $f \in S_k(\Gamma_0(N), \chi)$ of weight $k \geq 3$. For any $T>0$ and $\epsilon > 0$, we have
\[
	\mathfrak{D} =
		O_{f,\epsilon}\big(1+T^{k+\frac{3}{2} + \epsilon}\big).
\]
\end{theorem}

To prove Theorem~\ref{thm:D-estimate}, we represent $f(z)$ as a finite sum of holomorphic Poincar\'e
series, which we use for unfolding. The resulting objects are then understood using
Propositions~\ref{prop:spectral-average-appendix} and~\ref{prop:m-average}.

\subsection{An upper bound for \texorpdfstring{$\langle y^{\frac{k}{2}+\frac{1}{4}} f \overline{\theta}, \mu_j \rangle$}{an inner product}}

The cusp space $S_k(\Gamma_0(N), \chi)$ is finite-dimensional and spanned by Poincar\'e series $\{P_m\}_{m \geq 1}$ of the form
\[
  P_m(z) = \sum_{\gamma \in \Gamma_\infty \backslash \Gamma_0(N)}
  \overline{\chi(\gamma)} j(\gamma,z)^{-k} e(m \gamma z),
\]
in which $j(\gamma,z)$ is the usual $j$-invariant. The Sturm bound (\cite{sturm87}, or
see~\cite[Corollary~9.19]{Stein07} for more direct exposition) implies that our spanning set may restrict to
$m \ll_{k,N} 1$.

Consider the inner product
\begin{align} \label{eq:Poincare-basis-unfolding}
& \langle y^{\frac{k}{2}+\frac{1}{4}} P_m \overline{\theta}, \mu_j \rangle
=   \int_0^\infty \int_0^1
  \Im(z)^{\frac{k}{2}+\frac{1}{4}}
        \overline{\theta(z) \mu_j(z)} e(mz) \frac{dxdy}{y^2} \\
  &\qquad = \sum_{m=n_1+n_2} r_1(n_1) \overline{\rho_j(n_2)} \int_0^\infty
    y^{\frac{k}{2}-\frac{3}{4}} e^{-2\pi (n_1+m)y} W_{\frac{n_2 \kappa}{2 \vert n_2 \vert},it_j} (4\pi \vert n_2 \vert y) \frac{dy}{y}.
\end{align}
Let $G(n_1,n_2,m)$ denote the final integral above.  By~\cite[13.23.4]{DLMF}, $G$ may be written in terms of the ${}_2F_1$-hypergeometric function; using the Mellin--Barnes integral~\cite[15.6.6]{DLMF}, this implies
\[
G
  = \begin{cases}
  \displaystyle
  \frac{1}{2\pi i} \int_{(\Re w)}
    \frac{\Gamma(\frac{\kappa}{2} + it_j - w) \Gamma(\frac{\kappa}{2} -it_j -w)\Gamma(w)}
        {(4\pi n_2)^{\frac{k}{2}-\frac{3}{4}} \Gamma(\frac{1}{2}-w)}
        \Big(\frac{n_2}{n_1} \Big)^{w}  dw,
          &  n_2 > 0; \\[1.5em]
  \displaystyle
  \frac{1}{2\pi i} \int_{(\Re w)}
    \frac{\Gamma(\frac{\kappa}{2} + it_j - w) \Gamma(\frac{\kappa}{2} -it_j -w)\Gamma(w)}
        {(4\pi \vert n_2\vert )^{\frac{k}{2}-\frac{3}{4}} \Gamma(k-w)}
        \Big(\frac{\vert n_2 \vert}{m} \Big)^{w}  dw  ,
          & n_2 < 0,
    \end{cases}
\]
for any $\Re w \in (0, \frac{\kappa}{2} - \vert \Im t_j \vert)$.
The integrand decays exponentially outside of $\vert \Im w \vert \leq \vert t_j \vert$ by Stirling's
approximation.
In particular, $G(n_1,n_2,m) \ll_\kappa \vert n_2 \vert^{\Re w-
\frac{k}{2}-\frac{3}{4}}$ for $\vert t_j \vert \ll_\kappa 1$, which includes the Maass lifted holomorphic spectrum as well as any
potential exceptional eigenvalues $it_j \in \mathbb{R}$. The case $\vert t_j \vert \ll_\kappa 1$ therefore
gives $\langle y^{\frac{k}{2}+\frac{1}{4}} P_m \overline{\theta}, \mu_j \rangle
=O_{k,N}(1)$ by Proposition~\ref{prop:m-average} and dyadic subdivision (with
$\Re w$ near enough to $0$ to guarantee convergence of the sum).

Otherwise, for sufficiently large $c_\kappa$ and $\vert t_j \vert \geq c_\kappa$, Stirling gives the estimate
\begin{align*}
   \ll_{m,\kappa, \Re w} \vert \Im w - t_j \vert^{\frac{\kappa}{2} -\Re w -\frac{1}{2}}
  & \vert \Im w + t_j \vert^{\frac{\kappa}{2} -\Re w -\frac{1}{2}} \\
   \times & \vert \Im w \vert^{2\Re w - \frac{1}{2} - \kappa \, \delta_{[n_2 <0]}}
  e^{-\pi \vert t_j \vert} \vert n_2 \vert^{\Re w - \frac{k}{2}+\frac{3}{4}}
\end{align*}
for the integrand of $G(n_1,n_2,m)$ on the interval $\vert \Im w \vert \leq \vert t_j \vert$.
In the case $n_2 > 0$, integrating gives $G(n_1,n_2,m) \ll_{k,N} \vert t_j \vert^{\kappa-\frac{1}{2}} e^{-\pi \vert t_j \vert}$ for any choice of $\Re w$ (in part because $n_2 \leq m \ll_{k,N} 1$). For $n_2 < 0$, we have instead
\[
  G(n_1,n_2,m)
  \ll_{k,N,\Re w} \big( \vert t_j \vert^{-\frac{1}{2}} + \vert t_j \vert ^{\kappa - 2 \Re w -1} \big)
    e^{-\pi \vert t_j \vert} \vert n_2 \vert^{\Re w - \frac{k}{2} + \frac{3}{4}},
\]
valid for any $\Re w \in (0, \frac{\kappa}{2})$.
When $k \geq 3$ (as in Theorem~\ref{thm:D-estimate}), however, there is no
benefit in taking $\Re w$ outside $\Re w \in (0,\frac{\kappa}{2}-\frac{1}{4}]$.
We conclude that
\begin{align} \label{eq:inner-product-bound}
  \langle y^{\frac{k}{2}+\frac{1}{4}} P_m \overline{\theta}, \mu_j \rangle
  & \ll_{k,N,\Re w} \vert t_j \vert^{\kappa-\frac{1}{2}} e^{-\pi \vert t_j \vert}
      \sum_{0 \leq n < \sqrt{m}} \vert \rho_j(m-n^2) \vert  \\ \nonumber
  & \qquad +  \vert t_j \vert^{\kappa - 2 \Re w -1}
   e^{-\pi \vert t_j \vert} \!
    \sum_{n> \sqrt{m}} \frac{\vert \rho_j(m-n^2) \vert}{\vert m-n^2 \vert^{\frac{k}{2}-\frac{3}{4}- \Re w}}
\end{align}
when $\vert t_j \vert \geq c_\kappa$, for any $\Re w \in (0, \frac{\kappa}{2}-\frac{1}{4}] = (0,\frac{k}{2}-\frac{1}{2}]$.

\begin{remark} \label{rem:Maass-lift-poincare-example}
The computations involving $P_m(z)$ in this section can be used to provide explicit examples in which $\Sigma_{\mathrm{hol}}(s)$ admits a pole at $s=\frac{3}{4}$. For a concrete example, consider the Poincar\'e series $P_m(z)$ of (odd) weight $k$, level $N=576$, and character $\chi = (\frac{12}{\cdot})$. As noted in~\cite[\S{2.2}]{SerreStark77}, the space $S_{\frac{1}{2}}(\Gamma_0(576),\chi)$ is one-dimensional and spanned by $\theta_{\chi,1}(z)$. Let $\Theta_{\chi,1}(z)$ denote the Maass lift of $\theta_{\chi,1}$ to weight $\kappa$, scaled to have Fourier expansion
\begin{align} \label{eq:level-576-cuspform}
	\Theta_{\chi,1}(z) = \sum_{n \in \mathbb{Z}} \frac{\chi(n)}{(4\pi n^2)^{\frac{1}{4}}} W_{\frac{\kappa}{2},\frac{1}{4}}(4\pi n^2 y) e(n^2 x).
\end{align}
As in~\eqref{eq:Poincare-basis-unfolding}, we find that
\begin{align*}
& \langle y^{\frac{k}{2}+\frac{1}{4}} P_m \overline{\theta}, \Theta_{\chi ,1} \rangle
	= \int_0^\infty \int_0^1
		\Im(z)^{\frac{k}{2}+\frac{1}{4}} \overline{\theta(z) \Theta_{\chi,1}(z)} e(mz) \frac{dx dy}{y^2} \\
	&\qquad = \sum_{m=n_1+n_2^2} \frac{r_1(n_1) \overline{\chi(n_2)}}{(4\pi n_2^2)^{\frac{1}{4}}}
		\int_0^\infty
    y^{\frac{k}{2}-\frac{3}{4}} e^{-2\pi (n_1+m)y} W_{\frac{\kappa}{2},\frac{1}{4}} (4\pi n_2^2 y) \frac{dy}{y}.
\end{align*}

In the case $m=2$, the $m$-sum reduces to the case $n_1 = n_2^2 = 1$ and the Mellin--Barnes integral can be evaluated explicitly due to simplification in the hypergeometric functions in~\cite[13.23.4]{DLMF}. We conclude that
\[
	\langle y^{\frac{k}{2}+\frac{1}{4}} P_2 \overline{\theta}, \Theta_{\chi ,1} \rangle
		= 2^{\frac{11}{2}-\frac{5k}{2}} \pi^{\frac{1}{2}-\frac{k}{2}} \Gamma(k-1) \sin\Big(\frac{\pi}{4}(k+1)\Big),
\]
which is non-zero for $k \equiv 1 \bmod 4$.
In particular, $\Sigma_{\mathrm{hol}}$ has a non-zero pole at $s=\frac{3}{4}$ whenever the $h$-th Fourier coefficient of $\Theta_{\chi,1}$ is non-zero. This occurs, for example, in the case $h=1$.

By writing $P_2(z)$ as a linear combination of eigenforms in $S_k(\Gamma_0(576),\chi)$, we deduce the existence of some eigenform $f(z)$ for which $\Sigma_{\mathrm{hol}}(s)$ has a non-zero pole at $s =\frac{3}{4}$. We conjecture that such forms must be dihedral.
\end{remark}

\subsection{An upper bound for \texorpdfstring{$\mathfrak{D}$}{D}}

When $f(z) = P_m(z)$ and $T \ll_\kappa 1$, the estimate $\langle y^{\frac{k}{2}+\frac{1}{4}} P_m \overline{\theta}, \mu_j \rangle =O_{k,N}(1)$ implies that the spectral sum $\mathfrak{D}$ introduced in~\eqref{eq:D-definition} satisfies $\mathfrak{D} \ll_{k,N} 1$. Otherwise, for $f(z) = P_m(z)$ and $T \geq c_\kappa$, the inequality~\eqref{eq:inner-product-bound} implies that
\begin{align*}
 & \mathfrak{D} \ll_{k,N,\Re w} T^{2\kappa-1}
    \sum_{\vert t_j \vert \sim T} \bigg(
      \sum_{n < \sqrt{m}} \vert \rho_j(m-n^2) \vert \bigg)^2 e^{-\pi \vert t_j \vert} \\
 & \qquad \qquad \quad +  T^{2\kappa-4\Re w -2}
      \sum_{\vert t_j \vert \sim T} \bigg(\sum_{n> \sqrt{m}}
          \frac{\vert \rho_j(m-n^2) \vert}{\vert m-n^2 \vert^{\frac{k}{2}-\frac{3}{4}-\Re w}} \bigg)^2
      e^{-\pi \vert t_j \vert},
\end{align*}
where we omit the non-dominant cross-term.

By Cauchy--Schwarz and Proposition~\ref{prop:spectral-average-appendix}, the contribution from $n< \sqrt{m}$ satisfies the bound
\begin{align*}
&  T^{2\kappa -1} \sqrt{m} \sum_{n < \sqrt{m}}
  \sum_{\vert t_j \vert \sim T}\vert \rho_j(m-n^2) \vert^2 e^{-\pi \vert t_j \vert} \\
&\qquad \ll_\epsilon T^{2\kappa -1} \sqrt{m} \sum_{n < \sqrt{m}}
  \Big(\frac{T^{2-\kappa}}{m-n^2} + T^{-\kappa} (m-n^2)^{-\frac{1}{2}+\epsilon}\Big)
  \ll_{m,\kappa} T^{\kappa+1},
\end{align*}
which is admissible in Theorem~\ref{thm:D-estimate} since $m \ll_{k,N} 1$.

For the terms with $n> \sqrt{m}$, we split the sum at an unspecified $n$ for which $n^2 - m \sim M = M(T)$.
In the head $\mathfrak{D}_{\mathrm{head}}$ corresponding to $n^2 - m < M$, a worst-case bound over dyadic subintervals gives some $M_0 <M$ for which
\begin{align*}
\mathfrak{D}_{\mathrm{head}} & \ll
  T^{2\kappa-4\Re w-2} (\log M)^2 \sum_{\vert t_j \vert \sim T}
  \Big( \sum_{n^2-m \sim M_0} \frac{\vert \rho_j(m-n^2) \vert}{\vert m-n^2 \vert^{\frac{k}{2}-\frac{3}{4}-\Re w}} \Big)^2
   e^{-\pi \vert t_j \vert} \\
  & \ll T^{2\kappa-4\Re w-2} M^\epsilon M_0^{-k+2+2\Re w} \sum_{\vert t_j \vert \sim T}
  \Big(\sum_{n^2-m \sim M_0}  \frac{\vert \rho_j(m-n^2) \vert^2}{\cosh(\pi t_j)} \Big) \\
  & \ll T^{2\kappa-4\Re w-2} M^\epsilon M_0^{-k+2+2\Re w} \!\!\! \sum_{n^2-m \sim M_0} \!\!\!
        \Big( T^{2+\kappa} M_0^{-1} + T^{\kappa}M_0^{-\frac{1}{2}+\epsilon}\Big) \\
  & \ll T^{3\kappa-4\Re w-2} M^\epsilon M_0^{2\Re w-k+\frac{3}{2}}
    \big( T^2 +  M_0^{\frac{1}{2}+\epsilon}\big),
\end{align*}
in which we've applied Proposition~\ref{prop:spectral-average-appendix}. Here and for the rest of this section, all implicit constants may depend on $k$, $N$, $\epsilon$, and $\Re w$ (where that appears). Note that $M_0$ depends on $M$, $T$, and $\Re w$. To remove $M_0$ and $\Re w$ and produce a bound which depends only on $M$ and $T$, we vary $\Re w$ and find a worst-case $M_0$ in each case.
\begin{enumerate}
\item[a.] For $\Re w \leq \frac{k}{2}-1$, all $M_0$-powers are non-positive and the worst-case $M_0$ is $M_0 = 1$. We find $\mathfrak{D}_{\mathrm{head}} \ll T^{3\kappa -4 \Re w} M^\epsilon$. We optimize with $\Re w = \frac{k}{2} -1$ to produce $\mathfrak{D}_{\mathrm{head}} \ll T^{k+\frac{5}{2}} M^\epsilon$.
\item[b.] For $\Re w \geq \frac{k}{2}-\frac{3}{4}$, all $M_0$-powers are non-negative, so the worst-case $M_0$ is $M$ and so $\mathfrak{D}_{\mathrm{head}} \ll T^{3\kappa -4\Re w -2 +\epsilon} M^{2\Re w -k + \frac{3}{2} +2\epsilon} (T^2+ M^{\frac{1}{2}})$. If $M \gg T^2$, we benefit from taking $\Re w$ small; with $\Re w = \frac{k}{2}-\frac{3}{4}$, we find $\mathfrak{D}_{\mathrm{head}} \ll T^{k-\frac{1}{2}+\epsilon} M^{2\epsilon} (T^2 + M^{\frac{1}{2}})$. Conversely, if $M \ll T^2$, we benefit from $\Re w = \frac{k}{2}-\frac{1}{2}$, to produce $\mathfrak{D}_{\mathrm{head}} \ll T^{k+\frac{1}{2}+\epsilon} M^{\frac{1}{2}+2\epsilon}$.
\item[c.] For $\Re w \in [\frac{k}{2}-1, \frac{k}{2}-\frac{3}{4}]$, the
$M_0$-powers have mixed sign. A general upper bound is
$\mathfrak{D}_{\mathrm{head}} \ll T^{3\kappa -4 \Re w -2+\epsilon}
M^{2\epsilon} (T^2+M^{2\Re w -k + 2})$. When $T^2 \gg M^{2\Re w -k + 2}$, we benefit from taking $\Re w$ large and optimize with $\Re w = \min(\frac{k}{2}-\frac{3}{4},\frac{k}{2}-1 + \log_M T)$. This produces $\mathfrak{D}_{\mathrm{head}} \ll T^{\frac{5}{2}+k-4 \log_M T +\epsilon} M^{2\epsilon}$ for $M \gg T^4$ and $\mathfrak{D}_{\mathrm{head}} \ll T^{k+\frac{3}{2}+\epsilon} M^{2\epsilon}$ for $M \ll T^4$.
Conversely, if $T^2 \ll M^{2 \Re w - k+2}$, we must have $T \ll M^{\frac{1}{2}}$, which incentivizes $\Re w$ as small as possible, i.e.\ $\Re w = \frac{k}{2}-1 + \log_M T$. We find $\mathfrak{D}_{\mathrm{head}} \ll T^{\frac{5}{2}+k-4 \log_M T +\epsilon} M^{2\epsilon}$ as before.
\end{enumerate}
We conclude that $\mathfrak{D}_{\mathrm{head}}$ is $O(T^{k+\frac{1}{2}+\epsilon} M^{\frac{1}{2}+2\epsilon})$ when $M \ll T^2$, $O(T^{k+\frac{3}{2}+\epsilon} M^{2\epsilon})$ when $T^2 \ll M \ll T^4$, and $O(T^{k+\frac{5}{2}-4\log_M T + \epsilon} M^{2\epsilon})$ when $T^4 \ll M$.

We now consider the tail $\mathfrak{D}_{\mathrm{tail}}$ in which $n^2-m > M$. By Cauchy--Schwarz and Proposition~\ref{prop:m-average}, we have
\begin{align*}
& \sum_{n^2 -m \sim M} \frac{\vert \rho_j(m-n^2) \vert}{\vert m-n^2 \vert^{\frac{k}{2}-\frac{3}{4}-\Re w}}
  \ll M^{\Re w-\frac{k}{2}+\frac{3}{4}} \Big(\sum_{n^2-m \sim M} \vert \rho_j(m-n^2) \vert^2 \Big)^{\frac{1}{2}} \cdot M^{\frac{1}{4}} \\
& \,\, \ll M^{\Re w - \frac{k}{2}+1} \Big(\sum_{\ell \sim M} \vert \rho_j(-\ell) \vert^2 \Big)^{\frac{1}{2}} \!\!
  \ll M^{\Re w - \frac{k}{2}+1} (1+\vert t_j \vert)^{\frac{\kappa}{2}}
    \Big(1+\frac{\vert t_j \vert^{\frac{1}{2}}}{M^{\frac{1}{2}}}\Big) e^{\frac{\pi}{2} \vert t_j \vert}.
\end{align*}
The same result holds for the sum over all $n^2 -m > M$ by dyadic summation, provided $\Re w <\frac{k}{2}-1$. (Note that choice of $\Re w$ here is unrelated to our earlier choice of $\Re w$ in $\mathfrak{D}_{\mathrm{head}}$.) We now compute
\begin{align}
\mathfrak{D}_{\mathrm{tail}}
  & \ll T^{2\kappa-4\Re w -2} \sum_{\vert t_j \vert \sim T} \Big( \sum_{n^2-m > M}
    \frac{\vert \rho_j(m-n^2) \vert}{\vert m-n^2 \vert^{\frac{k}{2}-\frac{3}{4}-\Re w}}
   \Big)^2 e^{-\pi \vert t_j \vert} \\
  & \ll T^{2\kappa-4\Re w -2} \sum_{\vert t_j \vert \sim T}
      M^{2\Re w - k +2} T^\kappa \Big(1 + \frac{T}{M}\Big) \\
  & \ll T^{3\kappa-4\Re w} M^{2\Re w - k +2}  \Big(1 + \frac{T}{M}\Big), \label{eq:D_tail-bound}
\end{align}
by applying Proposition~\ref{prop:m-average} and the Weyl law.

Finally, we determine bounds for $\mathfrak{D}$. In the regime $1 \ll M \ll
T^2$, we are led by~\eqref{eq:D_tail-bound} to take $\Re w$ large; with $\Re w = \frac{k}{2}-1-\epsilon$, we produce
\[\mathfrak{D} \ll \mathfrak{D}_{\mathrm{head}} + \mathfrak{D}_{\mathrm{tail}}
 \ll T^{k+\frac{1}{2}+\epsilon} M^{\frac{1}{2}+2\epsilon}
  + T^{k+\frac{5}{2}+4\epsilon} M^{-2\epsilon} (1 + T/M),
\]
which is optimized at $M=T^2$ to produce $\mathfrak{D} \ll T^{k+\frac{5}{2}+\epsilon}$. Conversely, for $M \gg T^2$,~\eqref{eq:D_tail-bound} incentivizes $\Re w$ small; with $\Re w = \epsilon$, we produce
\[
  \mathfrak{D} \ll
		\begin{cases}
			T^{k+\frac{3}{2}+\epsilon} M^{2\epsilon}
				+ T^{3\kappa -4\epsilon} M^{2\epsilon -k +2}, & \quad T^2 \ll M \ll T^4, \\
			T^{k+\frac{5}{2}-4\log_M T + \epsilon} M^{2\epsilon}
				+ T^{3\kappa -4\epsilon} M^{2\epsilon -k +2}, & \quad M \gg T^4.
		\end{cases}
\]
As $k \geq 3$, we optimize with $M = T^{\frac{2k-3}{k-2}}$ in the case
$T^2 \ll M \ll T^4$ to produce $\mathfrak{D} \ll T^{k+\frac{3}{2} + \epsilon}$.
(The case $M \gg T^4$ is optimized with $M=T^4$ and does not improve this
estimate.)

This completes the proof of Theorem~\ref{thm:D-estimate} in the case $f=P_m$,
and the extension to general $f$ is straightforward.

Combining Theorem~\ref{thm:D-estimate} with Proposition~\ref{prop:spectral-average-appendix} via
Cauchy--Schwarz, we produce a useful average involving
$\rho_j(h) \langle y^{\frac{k}{2}+\frac{1}{4}} f \overline{\theta}, \mu_j \rangle$.

\begin{corollary} \label{cor:product-average} Fix $h> 0$ and $k \geq 3$. For any $T>0$ and $\epsilon > 0$, we have
\[\sum_{\vert t_j \vert \sim T} \vert \rho_j(h) \langle y^{\frac{k}{2}+\frac{1}{4}} f \overline{\theta}, \mu_j \rangle \vert
	= O_{f,\epsilon}\big(1+T^{2 +\epsilon}\big).
\]
\end{corollary}

\begin{remark} \label{rem:discrete-spectrum-conjecture}
Theorem~\ref{thm:D-estimate} and Corollary~\ref{cor:product-average} are not sharp. Assuming the generalized Lindel\"{o}f hypothesis and Ramanujan--Petersson conjecture, Lemma~\ref{lem:Duke-bound} implies that $\rho_j(-\ell) \ll \vert t_j \vert^{\frac{\kappa}{2}+\epsilon} e^{\frac{\pi}{2} \vert t_j \vert} \vert \ell \vert^{-\frac{1}{2}+\epsilon}$ for all $\epsilon > 0$. Stationary phase confirms that our estimate for $G(n_1,n_2,m)$ is relatively sharp and suggests that the absolute values in~\eqref{eq:inner-product-bound} can be moved outside the sum $n> \sqrt{m}$.  If the resulting sum demonstrates square-root cancellation, we would have
\[
  \langle y^{\frac{k}{2}+\frac{1}{4}} f \overline{\theta}, \mu_j \rangle
  \ll_{k,N,\epsilon} \vert t_j \vert^{\frac{k}{2}-\frac{3}{4}+\epsilon} e^{-\frac{\pi}{2} \vert t_j \vert}
\]
by combining Lemma~\ref{lem:Duke-bound} with~\eqref{eq:inner-product-bound} in the case $\Re w = \frac{k}{2}-\frac{1}{2}-\epsilon$.  This would imply that $\mathfrak{D} \ll_{f,\epsilon} T^{k+\frac{1}{2}+\epsilon}$ and that the sum in Corollary~\ref{cor:product-average} is $O_{f,\epsilon}(T^{\frac{3}{2}+\epsilon})$.
\end{remark}

\section{Sharp Cutoff Result} \label{sec:sharp-cutoff}

In this section, we apply Perron's formula (cf.~\cite[Lemma 3.12]{Titchmarsh86}) to $D_h(s)$ to study the partial sums of $A(n^2+h)$. By the definition of $D_h(s)$ from~\eqref{eq:D_h-definition} and Perron's formula, we have
\begin{align} \label{eq:perron}
\sum_{m^2 +h \leq X^2} \!\!\!  A(m^2+h)
  &= \frac{1}{4\pi i} \int_{1+\epsilon -iT}^{1+\epsilon+iT} \!
    D_h(\tfrac{s}{2}+\tfrac{1}{4}) X^s \frac{ds}{s}
  + O\bigg(\frac{X^{1+\epsilon}}{T}\bigg)
\end{align}
for fixed $\epsilon > 0$ and any $T>1$.

To understand the integral, we replace $D_h(s)$ with its spectral expansion
$\Sigma_{\mathrm{disc}}(s) + \Sigma_{\mathrm{res}}(s) + \Sigma_{\mathrm{cont}}(s)$ and shift the
line of integration. To justify this shift, we must quantify the growth of
$\Sigma_{\mathrm{disc}}$, $\Sigma_{\mathrm{res}}$, and $\Sigma_{\mathrm{cont}}$ in vertical strips.

\subsection{Growth of \texorpdfstring{$\Sigma_{\mathrm{disc}}$}{the discrete spectrum}} \label{subsec:discrete-spectrum-growth}

Recall from~\eqref{eq:discrete-spectrum} that the discrete spectral component of $D_h(s)$ equals
\begin{align*}
\Sigma_{\mathrm{disc}} :=
\frac{(4\pi)^{\frac{k}{2}+\frac{1}{4}}}{h^{s-1}}
\sum_j  \frac{\Gamma(s-\frac{1}{2}+it_j)\Gamma(s-\frac{1}{2}-it_j)}
       {\Gamma(s-\frac{k}{2}+\frac{1}{4})\Gamma(s+\frac{k}{2}-\frac{3}{4})}
   \rho_{j}(h) \langle y^{\frac{k}{2}+\frac{1}{4}} f \overline{\theta}, \mu_j \rangle.
\end{align*}

On the line $\Re s = \sigma$, Stirling's approximation gives the estimate
\begin{align*}
\Sigma_{\mathrm{disc}} \ll_k
\frac{h^{1-\sigma}}{\vert s \vert^{2\sigma-\frac{3}{2}}}
& \sum_j  \vert \rho_j(h)\langle y^{\frac{k}{2}+\frac{1}{4}} f \overline{\theta}, \mu_j \rangle \vert \\
& \quad \times \vert s + it_j \vert^{\sigma-1-\Im t_j} \vert s-it_j \vert^{\sigma-1+\Im t_j}
  e^{ \pi \vert s \vert - \pi \max(\vert s \vert , \vert t_j \vert)},
\end{align*}
showing that the mass of the discrete spectrum concentrates in $\vert t_j \vert < \vert s \vert$.
The contribution of Maass forms with $ \vert t_j \vert \ll_\kappa 1$ is $O_{k,N}(\vert s \vert^{-1/2})$, since there are $O_{k,N}(1)$ Maass forms of bounded spectral type by the Weyl law.

In the case $\vert t_j \vert \geq c_\kappa$,
we perform dyadic subdivision based on the size of $\min(\vert s + it_j \vert,\vert s - it_j \vert)$
and determine that
\begin{align*}
  \Sigma_{\mathrm{disc}}(s)
    &\ll_{f,h,\epsilon} \vert s \vert^{-\frac{1}{2}}
    + \vert s \vert^{\frac{1}{2}-\sigma} \sum_{0 \leq \ell \leq \log_2 \vert s \vert}
      (2^\ell)^{\sigma-1}
      (\vert s \vert -2^\ell)^{2+\epsilon} \\
    & \ll_{f,h,\epsilon} \vert s \vert^{\frac{3}{2}+\epsilon} \big( 1 + \vert s \vert^{1-\sigma} \big)
\end{align*}
by applying Corollary~\ref{cor:product-average}.

\subsection{Growth of \texorpdfstring{$\Sigma_{\mathrm{res}}$}{the residual spectrum}}
\label{subsec:residual-spectrum-growth}

The growth rate of the residual contribution $\Sigma_{\mathrm{res}}$ in vertical strips is obvious from Stirling's approximation and the explicit formulas~\eqref{eq:residual-squarefree-level} and~\eqref{eq:residual-general}. We conclude that $\Sigma_{\mathrm{res}}(s) \ll_{f,h} \vert s \vert^{-1/2}$.

\subsection{Growth of \texorpdfstring{$\Sigma_{\mathrm{cont}}$}{the continuous spectrum}}
  \label{subsec:continuous-spectrum-growth}

We recall from~\eqref{eq:continuous-spectrum} that the continuous spectrum's contribution towards $D_h(s)$ in $\Re s > \frac{1}{2}$ equals
\begin{align*}
\Sigma_{\mathrm{cont}}(s)
  := \frac{(4\pi)^{\frac{k}{2}-\frac{3}{4}}}{h^{s-1}} \sum_{\mathfrak{a}} &
  \int_{-\infty}^\infty \frac{\Gamma(s-\frac{1}{2} + it)\Gamma(s-\frac{1}{2}-it)}
        {\Gamma(s-\frac{k}{2}+\frac{1}{4})\Gamma(s+\frac{k}{2}-\frac{3}{4})} \\
&  \qquad
  \times \rho_\mathfrak{a}(h,\tfrac{1}{2}+it)
  \langle y^{\frac{k}{2}+\frac{1}{4}} f \overline{\theta},
      E_\mathfrak{a}^\kappa(\cdot,\tfrac{1}{2}+it; \chi) \rangle \, dt,
\end{align*}
in which $E_\mathfrak{a}^\kappa(z,v;\chi)$ is the Eisenstein series of weight $\kappa$, character $\chi \chi_{-1}^k$, and level $N$ at $\mathfrak{a}$ and $\rho_\mathfrak{a}(h,v)$ is its $h$th Fourier coefficient following~\eqref{eq:Fourier-Maass}.

To bound $\Sigma_{\mathrm{cont}}$, we need estimates for $\rho_\mathfrak{a}(h,v)$ and $\langle y^{\frac{k}{2}+\frac{1}{4}} f \overline{\theta} , E_\mathfrak{a}^\kappa(\cdot ,v; \chi) \rangle$ on the critical line $\Re v = \frac{1}{2}$.
These are given in the following lemmas.

\begin{lemma}
Let $\chi_{\kappa, h} = ( \frac{h(-1)^{\kappa-\frac{1}{2}}}{\cdot})$.
For any $\epsilon > 0$,
\[
  \rho_\mathfrak{a}(h,\tfrac{1}{2}+it)
  \ll_{h, \kappa, N, \epsilon}
  \frac{%
    L(\frac{1}{2} + 2it, \chi\chi_{\kappa,h})
  }{
    L^{(2N)}(1 + 4it, \chi^2)\Gamma(\frac{1}{2}+\frac{\kappa}{2}+it)
  }.
\]
\end{lemma}

\begin{proof}
  This result follows from recognizing the coefficients of $E_{\mathfrak{a}}^{\kappa}$
  as Dirichlet $L$-functions.
  The computations are tedious but very similar to the proofs of Proposition~1.2 and Corollary~1.3
  of~\cite{GoldfeldHoffstein85} (though the proofs there apply to a differently normalized
  Eisenstein series of level $4$, restrict to coefficients with square-free $m$, and don't evaluate the
  Archimedean integral).

  An alternative evaluation for general $m$ and with our normalization is summarized
  in~\cite[(2.1)-(2.2)]{LowryDuda17}.
  We note that the behavior of non-square-free coefficients $m$ differ from those of square-free
  coefficients by a finite Dirichlet correction factor depending on $m$.
  The generalization to higher level and non-trivial character is analogous.
\end{proof}

\begin{lemma}\label{lemma:continuous_inner_product}
  For any $\epsilon > 0$ and any singular cusp $\mathfrak{a}$,
  \begin{equation*}
    \langle y^{\frac{k}{2}+\frac{1}{4}} f \overline{\theta} , E_\mathfrak{a}^\kappa(\cdot ,\tfrac{1}{2}+it;\chi) \rangle
    \ll_{f, \epsilon}
    (1 + \lvert t \rvert)^{\frac{k}{2} + \epsilon}
    e^{-\frac{\pi}{2} \vert t_j \vert}.
  \end{equation*}
\end{lemma}

\begin{proof}
  The inner product $\langle y^{\frac{k}{2}+\frac{1}{4}} f \overline{\theta},
  E_\mathfrak{a}^\kappa(\cdot ,\overline{v};\chi) \rangle$ can be written as a Rankin--Selberg
  integral.
  Writing $\Gamma_\mathfrak{a}$ for the stabilizer of $\mathfrak{a}$ in $\Gamma_0(N)$, we recall
  that
  \[
    E_\mathfrak{a}^\kappa(z,v;\chi)
    =
    \sum_{\gamma \in \Gamma_{\mathfrak{a}} \backslash \Gamma_0(N)}
    \overline{\chi(\gamma)}
    J_\theta(\sigma_\mathfrak{a}^{-1} \gamma, z)^{-2\kappa}
    \Im(\sigma_{\mathfrak{a}}^{-1} \gamma z)^v,
  \]
  in which $\sigma_{\mathfrak{a}}$ is a scaling matrix for the cusp $\mathfrak{a}$. We take $\sigma_{\infty} = (\begin{smallmatrix} 1 & 0 \\ 0 & 1 \end{smallmatrix})$ and otherwise use the specific scaling matrix
\[
  \sigma_{\mathfrak{a}} =
    \left(\begin{matrix}
      \mathfrak{a} \sqrt{[N,w^2]} & 0 \\
      \sqrt{[N,w^2]} & 1/(\mathfrak{a} \sqrt{[N,w^2]})
    \end{matrix} \right)
\]
for the cusp $\mathfrak{a} = \frac{u}{w}$ to agree with~\cite[(2.3)]{DeshouillersIwaniec82}.

  A standard unfolding argument (following a change of variables
  $z \mapsto \sigma_\mathfrak{a} z$) then shows that
  \[
    \langle
      y^{\frac{k}{2}+\frac{1}{4}} f \overline{\theta},
      E_\mathfrak{a}^\kappa(\cdot ,\overline{v};\chi)
    \rangle
    =
    \iint_{\sigma_\mathfrak{a}^{-1}(\Gamma_\mathfrak{a} \backslash \mathbb{H})}
    y^{v+\frac{k}{2}+\frac{1}{4}} f_\mathfrak{a}(z) \overline{\theta_\mathfrak{a}(z)} \frac{dxdy}{y^2},
  \]
  in which  $\theta_\mathfrak{a} = \theta\vert_{\sigma_{\mathfrak{a}}}$ and $f_\mathfrak{a} = f\vert_{\sigma_{\mathfrak{a}}}$.
  We also note that $\sigma_\mathfrak{a}^{-1}(\Gamma_\mathfrak{a} \backslash \mathbb{H}) = \Gamma_\infty
  \backslash \mathbb{H}$.

  As in the standard Rankin--Selberg construction, this
  double integral has the Dirichlet series representation
  \begin{equation}\label{eq:continuous_dirichlet_series}
    \iint_{\Gamma_\infty \backslash \mathbb{H}}
    y^{v+\frac{k}{2}+\frac{1}{4}} f_\mathfrak{a}(z) \overline{\theta_\mathfrak{a}(z)} \frac{dxdy}{y^2}
    =
    \frac{\Gamma(v + \frac{k}{2} - \frac{3}{4})}{(4\pi)^{v + \frac{k}{2} - \frac{3}{4}}}
    \sum_{n \geq 1}
    \frac{a_\mathfrak{a}(n) \overline{r_\mathfrak{a}(n)}}
         {n^{v + \frac{k}{2} - \frac{3}{4}}},
  \end{equation}
  where $a_\mathfrak{a}(\cdot)$ and $r_\mathfrak{a}(\cdot)$ denote the Fourier coefficients of
  $f_\mathfrak{a}$ and $\theta_\mathfrak{a}$, respectively.

In the special case $\mathfrak{a} = \infty$ one can recognize the Dirichlet series in~\eqref{eq:continuous_dirichlet_series} in terms of the symmetric square $L$-function of $f$, so that
  \begin{equation*}
   \langle
     y^{\frac{k}{2}+\frac{1}{4}} f \overline{\theta}, E_\infty^\kappa(\cdot ,\tfrac{1}{2}+it;\chi)
   \rangle
   =
   \frac{%
     2\Gamma(\frac{\kappa}{2}-it)
     L(\frac{1}{2}-2it,\mathrm{Sym}^2 f)
   }{%
     (4\pi)^{\frac{\kappa}{2}-it}L(1-4it, \chi^2)
   },
  \end{equation*}
up to some factor addressing bad primes. In particular, in the case $\mathfrak{a} = \infty$, Lemma~\ref{lemma:continuous_inner_product} follows from the Phragm\'{e}n--Lindel\"{o}f convexity principle and Stirling's approximation.

More generally, Lemma~\ref{lemma:continuous_inner_product} reduces to convexity for the symmetric square $L$-function attached to (a twist of) the cusp form $f_\mathfrak{a}$. To see this, it suffices to show that $\theta_\mathfrak{a}$ has a Fourier expansion which resembles a twist of $\theta$ away from a finite set of exceptional primes $p \ll_N 1$. This can be verified through explicit computation.

For example, suppose that $\mathfrak{a} = \frac{u}{w}$ is $\Gamma_0(4)$-equivalent to the infinite cusp. Then there exists some matrix $\gamma \in \Gamma_0(4)$ so that $\infty = \gamma \cdot \mathfrak{a}$, which we may write in the form $\gamma =  (\begin{smallmatrix} a & b \\ w & -u \end{smallmatrix})$. By carefully tracking square roots in the relevant $j$-factors, we compute that
\begin{align} \label{eq:theta-explicit-equivalent-infinite}
\theta_\mathfrak{a}(z)
  = -i \, \epsilon_a^{-1} \Big(\frac{-w}{a}\Big) \Big(\frac{N}{(N,w^2)}\Big)^{\frac{1}{4}}
  \sum_{n \in \mathbb{Z}} e\Big(\frac{-n^2 b}{u}\Big) e\Big(\frac{n^2 N}{(N,w^2)} z \Big).
\end{align}
This implies that the Dirichlet series in~\eqref{eq:continuous_dirichlet_series} equals a certain symmetric square $L$-function away from bad primes, so the lemma holds whenever $\mathfrak{a}$ is $\Gamma_0(4)$-equivalent to $\infty$.

The casework for cusps which are $\Gamma_0(4)$-equivalent to $0$ or $\frac{1}{2}$ is suitably analogous, so we omit details.
\end{proof}

Combining these two lemmas, we find that
\begin{align} \label{eq:multiplied-continuous-bound}
  & \rho_\mathfrak{a}(h,\tfrac{1}{2}+it)
  \langle y^{\frac{k}{2}+\frac{1}{4}} f \overline{\theta} , E_\mathfrak{a}^\kappa(\cdot ,\tfrac{1}{2}+it;\chi) \rangle
  \\
  & \qquad \quad \ll_{f,h,\epsilon}
  (1+\lvert t \rvert)^{\frac{1}{2} - \frac{k}{2} +\epsilon}
  \cdot
  (1 + \lvert t \rvert)^{\frac{k}{2} + \epsilon}
  \ll_{f,h,\epsilon}
  (1+\lvert t \rvert)^{\frac{1}{2} + \epsilon}.
\end{align}
Consequently, for $\Re s = \sigma > \frac{1}{2}$, we have
\begin{align*}
  \Sigma_{\mathrm{cont}}
  & \ll_{f,h,\epsilon} \int_{-\infty}^\infty
  \bigg\lvert
    \frac{\Gamma(s-\frac{1}{2} + it)\Gamma(s-\frac{1}{2}-it)}
         {\Gamma(s-\frac{k}{2}+\frac{1}{4}) \Gamma(s+\frac{k}{2}-\frac{3}{4})}
  \bigg\rvert
  \lvert t \rvert^{\frac{1}{2} + \epsilon} dt
  \\
  & \ll \lvert s \rvert^{\frac{3}{2}-2\sigma} \!
    \int_{-\infty}^\infty
    \lvert s-it \rvert^{\sigma - 1}
    \lvert s+it \rvert^{\sigma - 1}
    \lvert t \rvert^{\frac{1}{2} + \epsilon}
    e^{-\pi \max( \lvert \Im s \rvert, \lvert t \rvert) + \pi \lvert \Im s \rvert} dt.
\end{align*}
The exponential terms effectively concentrate mass in $\lvert t \rvert < \lvert \Im s \rvert$, so that
\begin{align*}
  \Sigma_{\mathrm{cont}}(s)
  & \ll_{f,h,\epsilon} \lvert s \rvert^{\frac{1}{2}-\sigma}
  \int_{0}^{\lvert \Im s \rvert} \lvert s - it \rvert^{\sigma -1} \lvert t \rvert^{\frac{1}{2} + \epsilon} dt
  \ll_{f,h,\epsilon} \lvert s \rvert^{1-\sigma+\epsilon} + \lvert s \rvert^{1 + \epsilon},
\end{align*}
and hence $\Sigma_{\mathrm{cont}}(s) \ll_{f,h,\epsilon} \vert s \vert^{1+\epsilon}$ in $\Re s > \frac{1}{2}$.

\begin{remark} \label{rem:continuous-spectrum-conjecture}
  The bound $\Sigma_{\mathrm{cont}}(s) \ll_{f,h,\epsilon}  \vert s \vert^{1+\epsilon}$ suffices for our purposes but is by no means sharp. Under the generalized Lindel\"{o}f hypothesis, the upper bound~\eqref{eq:multiplied-continuous-bound} improves to $O_{f,h,\epsilon}((1+\vert t \vert)^{-\frac{1}{2}+\epsilon})$ and it would follow that $\Sigma_{\mathrm{cont}}(s) \ll_{f,h,\epsilon}  \vert s \vert^{\epsilon}$ in the half-plane $\Re s > \frac{1}{2}$.
\end{remark}

\subsection{Contour Shifting} \label{sec:contour-shifting}

The growth estimates from~\S{\ref{subsec:discrete-spectrum-growth}},~\S{\ref{subsec:residual-spectrum-growth}}, and~\S{\ref{subsec:continuous-spectrum-growth}} imply that the growth of $D_h(s)$ in vertical strips in $\Re s > \frac{1}{2}$ is dominated by that of $\Sigma_{\mathrm{disc}}$. Hence $D_h(s) \ll  \vert s \vert^{\frac{3}{2} + \epsilon}(1+ \vert s \vert^{1-\Re s})$ in a fixed vertical strip in $\Re s > \frac{1}{2}$,
where here and throughout~\S{\ref{sec:contour-shifting}} all implicit constants are allowed to depend on $f$, $h$, $\epsilon$, and $\Re s$ (where it appears).

In particular, on the line $\Re s = \frac{1}{2} + \epsilon$, it follows that
\begin{align} \label{eq:D_h-critical-line}
	D_h(\tfrac{s}{2}+\tfrac{1}{4}) X^s/s
	\ll \vert s \vert^{1+ \epsilon} X^{\frac{1}{2}+\epsilon}.
\end{align}
Note also that $D_h(\frac{s}{2}+\frac{1}{4}) X^s/s \ll X^{1+\epsilon}/\vert s \vert$ on the line $\Re s = 1+\epsilon$, by absolute convergence of the Dirichlet series.
In the vertical strip $\Re s \in ( \frac{1}{2}+ \epsilon, 1+\epsilon)$ between these estimates, $D_h(\frac{s}{2}+ \frac{1}{4}) X^s/s$ is meromorphic, with simple poles at most at $s=1$ (from $\Sigma_{\mathrm{hol}}$ and $\Sigma_{\mathrm{res}}$) and each real $s = \frac{1}{2} \pm 2it_j$ corresponding to an exceptional eigenvalue (from $\Sigma_{\mathrm{disc}}$, excluding $\Sigma_{\mathrm{hol}}$ by convention).
The Weyl law implies that exceptional eigenvalues, if they exist, are limited in number by $O_{k,N}(1)$.


Away from these finitely many poles, the convexity principle implies that
$D_h(\frac{s}{2}+\frac{1}{4}) X^s/ s \ll \vert s \vert^{1+ \epsilon} X^{\frac{1}{2}+\epsilon} + X^{1+\epsilon} / \vert s \vert$ in the vertical strip $\Re s \in ( \frac{1}{2}+ \epsilon, 1+\epsilon)$. We conclude from~\eqref{eq:discrete-spectrum}, \eqref{eq:discrete-spectrum-residue}, and~\eqref{eq:residual-general} that
\begin{align}
\begin{split} \label{eq:first-contour-shift}
 \frac{1}{4\pi i} \int_{1+\epsilon -iT}^{1+\epsilon+iT}
   D_h(\tfrac{s}{2}+\tfrac{1}{4})  X^s \frac{ds}{s} & \\
 \qquad = (b_{f,h} + c_{f,h}) X + \mathfrak{R}_E
   & + \frac{1}{4\pi i} \int_{\frac{1}{2}+\epsilon -iT}^{\frac{1}{2}+\epsilon+iT} \!
    D_h(\tfrac{s}{2}+\tfrac{1}{4}) X^s \frac{ds}{s} \\
&   + O\Big(\frac{X^{1+\epsilon}}{T} + T^{1+\epsilon} X^{\frac{1}{2}+\epsilon} \Big) ,
\end{split}
\end{align}
in which $\mathfrak{R}_E$ is the sum over possible residues arising from exceptional eigenvalues, given explicitly by
\begin{align*}
\mathfrak{R}_E :=
&  (4\pi)^{\frac{k}{2}+\frac{1}{4}} h^{\frac{1}{2}} X^\frac{1}{2}
    \sum_{it_j \in \mathbb{R}} \!
      \frac{(X^2/h)^{it_j} \Gamma(2i t_j) \rho_j(h) \langle y^{\frac{k}{2}+\frac{1}{4}} f \overline{\theta},\mu_j \rangle}
      {(\frac{1}{2}+2it_j) \Gamma(\frac{3}{4}-\frac{k}{2}+it_j) \Gamma(\frac{k}{2}-\frac{1}{4} + it_j)} \\
& \quad + (4\pi)^{\frac{k}{2}+\frac{1}{4}} h^{\frac{1}{2}} X^\frac{1}{2}
    \sum_{it_j \in \mathbb{R}} \!
      \frac{(X^2/h)^{-it_j} \Gamma(-2i t_j) \rho_j(h) \langle y^{\frac{k}{2}+\frac{1}{4}} f \overline{\theta},\mu_j \rangle}
      {(\frac{1}{2}-2it_j) \Gamma(\frac{3}{4}-\frac{k}{2}-it_j) \Gamma(\frac{k}{2}-\frac{1}{4} - it_j)}.
\end{align*}

The contribution of the continuous spectrum $\Sigma_{\mathrm{cont}}$ in $D_h(\frac{s}{2}+\frac{1}{4})$ on the line $\Re s = \frac{1}{2}+\epsilon$ is $O(X^{\frac{1}{2}+\epsilon} T^{1+\epsilon})$ following~\S{\ref{subsec:continuous-spectrum-growth}}. For the components of $D_h(\frac{s}{2}+\frac{1}{4})$ coming from the residual and discrete spectra, we shift the vertical contour farther left, to the line $\Re s = \epsilon$. In $\Sigma_{\mathrm{disc}}$, this shift passes a line segment of \emph{non-exceptional} spectral poles, which contributes a finite sum of residues $\mathfrak{R}$ of the form
\begin{align*}
& 2 (4\pi)^{\frac{k}{2}+\frac{1}{4}} h^{\frac{1}{2}} X^\frac{1}{2} \Re \!
\Big(
\sum_{0 \leq t_j \leq 2T} \!
  \frac{(X^2/h)^{it_j} \Gamma(2i t_j) \rho_j(h)
      \langle y^{\frac{k}{2}+\frac{1}{4}} f \overline{\theta},\mu_j \rangle}
  {(\frac{1}{2}+2it_j) \Gamma(\frac{3}{4}-\frac{k}{2}+it_j) \Gamma(\frac{k}{2}-\frac{1}{4} + it_j)} \Big) \\
& \,\,\, = (4\pi)^{\frac{k}{2}-\frac{1}{4}} h^{\frac{1}{2}} X^\frac{1}{2} \Im \!
\Big(
\sum_{0 \leq t_j \leq 2T} \!\!\!\!
\frac{(4X^2/h)^{it_j} \rho_j(h) \langle y^{\frac{k}{2}+\frac{1}{4}} f \overline{\theta}, \mu_j \rangle}{t_j} \big(1+ O_k(\tfrac{1}{t_j})\big)\! \Big).
\end{align*}

We treat $\mathfrak{R}$ as an error term and apply Corollary~\ref{cor:product-average} to conclude that
\begin{align} \label{eq:discrete-residue-bound}
\mathfrak{R} & \ll
  X^{\frac{1}{2}} \! \sum_{\ell \leq \log_2 T} \sum_{\vert t_j \vert \sim 2^{-\ell} T}
  \frac{\vert \rho_j(h) \langle y^{\frac{k}{2}+\frac{1}{4}} f \overline{\theta},\mu_j \rangle \vert}
        {\vert t_j \vert} \\ \nonumber
& \ll X^{\frac{1}{2}} \! \sum_{\ell \leq \log_2 T} (T/2^\ell)^{1 +\epsilon}
\ll X^{\frac{1}{2}} T^{1 +\epsilon}.
\end{align}

Following~\eqref{eq:perron},~\eqref{eq:first-contour-shift},~\eqref{eq:discrete-residue-bound}, and the estimate $O(X^{\frac{1}{2}+\epsilon} T^{1+\epsilon})$ for the shifted continuous spectrum, we have
\begin{align*}
 & \sum_{m^2 +h \leq X^2} A(m^2+h)  = (b_{f,h} + c_{f,h}) X + \mathfrak{R}_E \\
& \qquad \quad   + \frac{1}{4\pi i} \int_{\epsilon -iT}^{\epsilon+iT}
    (\Sigma_{\mathrm{res}} + \Sigma_{\mathrm{disc}})(\tfrac{s}{2}+\tfrac{1}{4}) X^s \frac{ds}{s}
 + O\bigg(\frac{X^{1+\epsilon}}{T}
    + T^{1 + \epsilon} X^{\frac{1}{2}+\epsilon}
    \bigg).
\end{align*}

The contour integral over $\Re s = \epsilon$ is $O(T^{\frac{9}{4} + \epsilon} X^{\epsilon})$ following the upper bounds $\Sigma_{\mathrm{disc}}(s) \ll \vert s \vert^{\frac{9}{4} +\epsilon}$ and $\Sigma_{\mathrm{res}}(s) \ll \vert s \vert^{-\frac{1}{2}}$ on the line $\Re s = \frac{1}{4}+\epsilon$. By taking $T = X^{1/4}$, we optimize our collective errors to size $O(X^{\frac{3}{4} + \epsilon})$. Since $\mathfrak{R}_E \ll X^{1/2 + \Theta} \ll X^{39/64}$ by comments in~\S{\ref{subsec:discrete-spectrum}, the potential contribution of $\mathfrak{R}_E$ may be ignored. This completes the proof of our main arithmetic result.

\begin{theorem} \label{thm:main-theorem}
For $k \geq 3$, $h > 0$, and any $\epsilon > 0$, we have
\[\sum_{m^2 +h \leq X^2} A(m^2+h) =
  (b_{f,h} + c_{f,h}) X + O_{f,h,\epsilon}\big(X^{\frac{3}{4} + \epsilon}\big),
\]
in which $b_{f,h} =0$ in many cases following Remark~\ref{rem:holomorphic-main-term} and $c_{f,h}=0$ in many cases following Remark~\ref{rem:general-main-term}.
\end{theorem}

\begin{remark}
The heuristic and conditional improvements to Corollary~\ref{cor:product-average} noted in Remark~\ref{rem:discrete-spectrum-conjecture} would imply that $\Sigma_{\mathrm{disc}} \ll \vert s \vert^{1+\epsilon} + \vert s \vert^{2-\Re s +\epsilon}$ and that $\mathfrak{R} \ll X^{\frac{1}{2}} T^{\frac{1}{2}+\epsilon}$. In addition, our bound for the shifted continuous integral would improve to $O(X^{\frac{1}{2}+\epsilon} T^\epsilon)$ under the generalized Lindel\"{o}f hypothesis following Remark~\ref{rem:continuous-spectrum-conjecture}. Optimizing errors with $T=X^{1/3}$ would improve the error in Theorem~\ref{thm:main-theorem} to $O(X^{\frac{2}{3}+\epsilon})$, which is comparable to Bykovskii's work on the divisor function~\cite{Bykovskii87}.  Error terms of size $O(X^{\frac{1}{2}+\epsilon})$ are conjectured to hold in both problems.
\end{remark}

\appendix
\section{A spectral average of Fourier coefficients}\label{sec:appendix}

The purpose of this appendix is to prove Proposition~\ref{prop:spectral-average-appendix}, a
strengthened version of~\cite[Lemma 5]{Blomer08}.
The main idea and strategy goes back to Kuznetsov~\cite{Kuznetsov80} which reduces the problem to a
bound on Kloosterman sums and an oscillatory integral.
The main improvement over~\cite[Lemma 5]{Blomer08} comes from studying further the
oscillatory integral in
Proposition~\ref{prop:appendix-oscill-int}.

\subsection*{Half-integral Kloosterman sums}%

Let $\ell \in \mathbb{Z}$ be an odd integer and let $\chi$ be a Dirichlet character of modulus $N$
for some $N \in \mathbb{N}$.
For $m,n \in \mathbb{Z}$ and $c \in \mathbb{N}$ with $[4,N] \mid c$, we define the Kloosterman sum
\begin{equation}
	K_{\ell}(m,n;c; \chi)
  :=
  \sum_{ad \equiv 1 \bmod (c)}
  \epsilon_d^{\ell} \, \overline{\chi(d)} \legendre{c}{d} e\left( \frac{ma+nd}{c} \right),
	\label{eq:appendix-Kloosterman-def}
\end{equation}
where $\epsilon_d$ is $1$ or $i$ as with signs of Gauss sums and $\legendre{c}{d}$ is the extended
Kronecker symbol as in \S\ref{sec:triple-inner-product}.
We will require the following bound, whose proof we defer until \S\ref{ssec:technical}.

\begin{proposition}\label{prop:appendix-Kloos-bound}
  Let $m,n \in \mathbb{Z}$ and $c,N \in \mathbb{N}$ be integers with $[4,N] \mid c$.
  Let $\chi$ be a Dirichlet character of modulus $N$.
  Then, for an odd integer $\ell \in \mathbb{Z}$, we have the bound
	\begin{equation*}
	\lvert K_{\ell}(m,n;c;\chi) \rvert \le 4 \tau(c) (m,n,c)^{1/2} c^{1/2} N^{1/2}.
	\end{equation*}
\end{proposition}

\subsection*{An oscillatory integral}%

The following oscillatory integral appears in the Kuznetsov pre-trace formula:
\begin{equation}\label{eq:appendix-I-int-def}
	I_{\kappa}(\omega,t) = -2i \omega \int_{-i}^{i} K_{2it}\left( \omega q \right) q^{\kappa-1} dq,
\end{equation}
for $\kappa,\omega,t \in \mathbb{R}$ with $\omega > 0$ and where the integral $\int_{-i}^{i}$ is
taken along the unit circle in positive/anti-clockwise direction.
We will require the following bound.

\begin{proposition}\label{prop:appendix-oscill-int}
  For $\kappa \in ]-2,2[$ and $T \ge 0$, we have
	\begin{equation}\label{eq:appendix-G-definition}
		G_{\kappa}(\omega,T)
    =
    \int_0^{T} t I_{\kappa}(\omega,t) dt \ll
    \begin{cases}
      \omega^{1/2}, & \omega \ge 1, \\
      \omega(1+|\log(\omega)|), & \omega \le 1,
    \end{cases}
	\end{equation}
	where the implied constant depends only on $\kappa$.
\end{proposition}

The special case $\kappa=0$ was first treated by Kuznetsov~\cite[\S 5]{Kuznetsov80}.
Other special cases and slight variants may also be found in the literature: the case $\kappa=1$ was
treated by Humphries~\cite[\S6]{Humphries18} and the cases $\kappa = \pm \frac{1}{2}$ with alternate $t$
averages were considered by Ahlgren--Andersen~\cite[\S 3]{AhlgrenAndersen18},
Andersen--Duke~\cite[\S4]{AndersenDuke20}, and Blomer \cite[Lemma 5]{Blomer08}.

Our general strategy of proof is the same as in the former four references.
However, a crucial point in the analysis will be the vanishing of a particular integral
(see~\eqref{eq:appendix-second-int-completion}).
This has been observed by Andersen--Duke~\cite{AndersenDuke20} in the case $\kappa=\frac{1}{2}$,
though little attention has been brought to this serendipity.
We also defer this proof until \S\ref{ssec:technical}.

\subsection*{A spectral average}%

Denote by $\upsilon_{\theta}$ be the weight $\frac{1}{2}$ $\theta$-multiplier system.
Let $N \in \mathbb{N}$ be an integer divisible by $4$ and $\chi$ a Dirichlet character of modulus $N$.
Further, let $\ell \in \mathbb{Z}$ an odd integer such that $\chi \upsilon_{\theta}^{\ell}$ is a
multiplier system of weight $\kappa \in \{\frac{1}{2},\frac{3}{2}\}$, i.e.\@ $\ell \in \{1,3\}$ if
$\chi$ is even or $\ell \in \{\pm 1\}$ if $\chi$ is odd.

Let $\{\mu_j\}_{j \ge 0}$ together with $\{E^{\kappa}_{\mathfrak{a}}(\cdot,w)\}_{\mathfrak{a},
\Re(w)=\frac{1}{2}}$ denote a complete $\Delta_{\kappa}$-eigenpacket for the $L^2$-space of
functions $f: \mathbb{H} \to \mathbb{C}$ satisfying
\begin{itemize}
  \item $f(\gamma z) = \chi(\gamma)\upsilon_{\theta}(\gamma)^{\ell}
  \left(\tfrac{j(\gamma,z)}{|j(\gamma,z)|}\right)^{\kappa}f(z)$, for all $\gamma \in \Gamma_0(N)$
  and $z \in \mathbb{H}$,

	\item $f$ is of at most moderate growth at the cusps of $\Gamma_0(N)$.
\end{itemize}
Here, the inner product is given by
\begin{equation*}
  \langle f,g \rangle
  =
  \int_{\Gamma_0(N) \backslash \mathbb{H}} f(z) \overline{g(z)} \frac{dxdy}{y^2}.
\end{equation*}
We suppose that the eigenpacket is normalized such that the spectral expansion
\begin{equation*}
  f(z)
  =
  \sum_j \langle f, \mu_j \rangle \mu_j(z)
  +
  \frac{1}{4\pi i} \sum_{\mathfrak{a}}
  \int_{\left(\frac{1}{2}\right)}
  \langle f, E_{\mathfrak{a}}^{\kappa}(\cdot,w) \rangle E^{\kappa}_{\mathfrak{a}}(z,w) dw
\end{equation*}
holds in $L^2$.
In particular, the discrete part $\{\mu_j\}_{j \ge 0}$ is $L^2$-normalized.
We write $t_j$ for the spectral parameter of $\mu_j$, which is characterized (up to sign) by the
equation $(\frac{1}{4}+t_j^2+\Delta_{\kappa})u_j=0$.
We have either $t_j \in \mathbb{R}$ or $t_j \in [-\frac{1}{4}i,\frac{1}{4}i]$, see~\cite[Satz
5.4]{Roelcke66}.
As in~\eqref{eq:Fourier-Maass}, we denote by $\rho_j(n)$, respectively $\rho_{\mathfrak{a}}(n,w)$,
the Fourier coefficients (at the cusp $\infty$) of $\mu_j$, respectively
$E^{\kappa}_{\mathfrak{a}}(\cdot,w)$, for $n \in \mathbb{Z} \backslash \{0\}$.
We have the following pre-trace formula,
see~\cite[Lemma 3]{Proskurin05}\footnote{\label{foot:convergence}The bounds established in
the proof of Proposition~\ref{prop:appendix-Kloos-bound}
guarantee the absolute convergence of both sides as $\sigma \to
1^+$.},~\cite[Lemma 3]{Blomer08}, or~\cite[Propositions 3.6.8, 3.6.9]{SteinerPhDthesis}\footnote{See
footnote~\ref{foot:convergence}.}.

\begin{proposition}[Kuznetsov pre-trace formula]\label{prop:appendix-pre-trace}
  Let $m,n \in \mathbb{Z} $ two integers satisfying $mn> 0$. Denote by $\pm$ the sign of $m$
  (respectively $n$).
  Then, for any $t \in \mathbb{R}$, we have
	\begin{multline}\label{eq:Kuz-pre-equal}
		\sum_{j} \frac{\sqrt{mn} }{\cosh(\pi(t-t_j))\cosh(\pi(t+t_j))} \overline{\rho_j(m)} \rho_j(n)\\
		+
    \frac{1}{4\pi} \sum_{\mathfrak{a}} \int_{-\infty}^{\infty}
    \frac{\sqrt{mn} }{\cosh(\pi(t-r))\cosh(\pi(t+r))}
    \overline{\rho_{\mathfrak{a}}(m;\tfrac{1}{2}+ir)}
    \rho_{\mathfrak{a}}(n;\tfrac{1}{2}+ir) dr
    \\
		=
    \frac{\lvert \Gamma(1\mp\frac{\kappa}{2}+it) \rvert^2}{4\pi^3}
    \Biggl\{
      \delta_{m,n}
      +
      \sum_{c \equiv 0 \, (N)} \frac{K_{\ell}(m,n;c;\chi)}{c}
      I_{\pm\kappa}\left( \frac{4 \pi \sqrt{mn}}{c} ,t \right)
    \Biggr\},
	\end{multline}
  where $K_{\ell}(m,n;c;\chi)$ is as in~\eqref{eq:appendix-Kloosterman-def} and
  $I_{\kappa}(\omega,t)$ is as in~\eqref{eq:appendix-I-int-def}.
\end{proposition}

We are now ready to state and prove the primary proposition.

\begin{proposition}\label{prop:appendix-Fourier-inequality}
  Let $m \in \mathbb{Z}\backslash\{0\}$ be a non-zero integer and let $\pm$ denote its sign.
  Then, for any $T \ge 1$, we have
	\begin{multline*}
		\sum_{\lvert t_j \rvert \le T} \frac{\lvert m \rvert}{\cosh(\pi t_j)}
    \max\{1,\lvert t_j \rvert^{\kappa}\}^{\pm 1} \lvert \rho_j(m) \rvert^2
    \\
		+
    \frac{1}{4 \pi} \sum_{\mathfrak{a}} \int_{-T}^{T}
    \frac{\lvert m \rvert}{\cosh(\pi r)} \max\{1,\lvert r \rvert^{\kappa}\}^{\pm 1}
    \lvert \rho_{\mathfrak{a}}(m; \tfrac{1}{2}+ir) \rvert^2 dr
    \\
		\ll T^2+(m,N)^{1/2}  \frac{\lvert m \rvert^{1/2}}{N^{1/2}} \lvert mN \rvert^{o(1)}.
	\end{multline*}
\end{proposition}

\begin{remark}
  A more careful analysis, as in~\cite[\S 5]{Kuznetsov80} or~\cite[\S 3]{AhlgrenAndersen18}, turns
  the upper bound into the asymptotic
  \begin{equation*}
    \frac{1}{4\pi^2}T^2+O\left(T \log(2T)+(m,N)^{1/2}
    \frac{|m|^{1/2}}{N^{1/2}} |mN|^{o(1)}\right).
  \end{equation*}
  We leave the details to the interested reader.
\end{remark}

\begin{remark}
  Proposition~\ref{prop:spectral-average-appendix} follows from
  Proposition~\ref{prop:appendix-Fourier-inequality} after applying normalized
  Maass weight increasing and decreasing operators, which are an isometry on
  the space generated by the spectrum $\lambda \geq \frac{1}{4}$ (for $\kappa
  \not \in \mathbb{Z}$).
\end{remark}

\begin{proof}[Proof of Proposition~\ref{prop:appendix-Fourier-inequality}]
  We apply the Kuznetsov pre-trace formula~\ref{prop:appendix-pre-trace} with $m=n$, mutliply the
  equality with $2 \pi t|\Gamma(1\mp \frac{\kappa}{2}+it)|^{-2}$, and integrate $t$ from $0$ to $T$ to
  arrive at:
	\begin{multline*}
		\sum_{\lvert t_j \rvert \le T} \frac{\lvert m \rvert}{\cosh(\pi t_j)}  \lvert \rho_j(m) \rvert^2 H_{\pm \kappa}(t_j,T)
    \\
		+
    \frac{1}{4 \pi} \sum_{\mathfrak{a}} \int_{-T}^{T} \frac{\lvert m \rvert}{\cosh(\pi r)}
    \lvert \rho_{\mathfrak{a}}(m; \tfrac{1}{2}+ir) \rvert^2 H_{\pm \kappa}(r,T)  dr
    \\
		=
    \frac{1}{4 \pi^2} T^2
    +
    \sum_{c \equiv 0 \, (N)} \frac{K_{\ell}(m,m;c;\chi)}{c} G_{\pm\kappa}
    \left( \frac{4 \pi |m|}{c} ,T \right),
	\end{multline*}
	where $G_{\kappa}(\omega,T)$ is as in~\eqref{eq:appendix-G-definition} and
  \begin{align*}
		H_{\kappa}(r,T)
    &=
    2 \pi \int_{0}^T \frac{t}{\lvert \Gamma(1-\frac{\kappa}{2}+it) \rvert^2}
    \frac{\cosh(\pi r)}{\cosh(\pi(t+r))\cosh(\pi(t-r))} dt
    \\
		&=
    2 \int_0^T \frac{\pi t}{ \lvert \Gamma(1-\frac{\kappa}{2}+it) \rvert^2 \cosh(\pi t)}
    \frac{\cosh(\pi t)\cosh(\pi r)}{\sinh(\pi t)^2+\cosh(\pi r)^2} dt.
	\end{align*}
  The second expression shows clearly that $H(t,T) \ge 0$ for $T \ge 1$ and $r \in \mathbb{R}$ or $r
  \in [-\frac{1}{4}i,\frac{1}{4}i]$.
  We furthermore claim that $H_{\kappa}(r, T) \gg (1+|r|)^{\kappa}$ in the indicated ranges assuming
  additionally that $|r|\le T$.
  Indeed, Stirling's approximation for the Gamma function yields
	\begin{equation*}
    \frac{\pi t}{ \lvert \Gamma(1-\frac{\kappa}{2}+it) \rvert^2 \cosh(\pi t)} = t^{\kappa}(1+O(t^{-1})),
	\end{equation*}
	for $t \gg 1$. Hence, the integrand from $[\frac{1}{2},1]$ is $\gg 1$ if $\lvert r \rvert\le 1$
  and the integrand from $[\lvert r \rvert-\frac{1}{2},\lvert r \rvert]$ is $\gg \lvert r
  \rvert^{\kappa}$ if $1 \le \lvert r \rvert \le T$.
  This proves the desired lower bound.
  For the upper bound, we use Propositions~\ref{prop:appendix-Kloos-bound}
  and~\ref{prop:appendix-oscill-int}.
  We have
	\begin{multline*}
		\sum_{c \equiv 0 \, (N)} \frac{|K_{\ell}(m,m;c;\chi)|}{c} \min\left\{ \left(\frac{|m|}{c}\right)^{1/2}, \left(\frac{|m|}{c}\right)^{1+o(1)} \right\}
		\\ \ll (m,N)^{1/2}N^{o(1)} \sum_{e \, \mid \, m} e^{o(1)} \sum_{c=1}^{\infty} \frac{1}{c^{1/2-o(1)}} \min\left\{ \left(\frac{|m|}{Nc}\right)^{1/2}, \left(\frac{|m|}{Nc}\right)^{1+o(1)} \right\}  \\
		\ll (m,N)^{1/2}  \frac{|m|^{1/2}}{N^{1/2}} |mN|^{o(1)}.
    \qedhere
	\end{multline*}
\end{proof}

\subsection{Technical proofs}\label{ssec:technical}

Finally, we give the technical details of the proofs we omitted above.

\subsection*{Proof of Proposition~\ref{prop:appendix-Kloos-bound}}

It is useful to introduce the related Dirichlet-twisted Sali\'e sums
\begin{equation*}
  S(m,n;c;\chi)
  :=
  \sum_{ad \equiv 1 \bmod (c)}
  \overline{\chi(d)} \legendre{d}{c} e\left( \frac{ma+nd}{c} \right),
\end{equation*}
where $m,n \in \mathbb{Z}$, $c \in \mathbb{N}$, $N \mid c$, and $v_2(c) \neq 1$.
We split the proof of Proposition~\ref{prop:appendix-Kloos-bound} into several
smaller Lemmas.

Lemma 2 of~\cite{Iwaniec87} showed a twisted multiplicativity relation between Kloosterman and
Sali\'e sums when $\chi$ is trivial.
With small adjustments, we obtain the following lemma.

\begin{lemma}\label{lem:appendix-Kloost-twist-mult}
  Let $r,s \in \mathbb{N}$ be two relatively prime integers with $N \mid rs$ and $4 \mid s$.
  Suppose the Dirichlet character $\chi$ modulo $N$ factors as $\chi_r$ modulo $(N,r)$ times
  $\chi_s$ modulo $(N,s)$, and suppose $\overline{r},\overline{s} \in \mathbb{Z}$ are integers
  satisfying $\overline{r}r+\overline{s}s=1$.
  Then, we have
	\begin{equation}\label{eq:twist-Kloost-twist-relation}
		K_{\ell}(m,n;rs;\chi)
    =
    S(m\overline{s},n\overline{s};r;\chi_r) K_{\ell+r-1}(m\overline{r},n\overline{r};s;\chi_s).
	\end{equation}
\end{lemma}

\begin{proof}
  We write $d=x\overline{r}r +y \overline{s}s$, where $y$ runs over a representative system modulo
  $r$ with $(y,r)=1$ and likewise $x$ modulo $s$ with $(x,s)=1$.
  Then, we have $\epsilon_d= \epsilon_x$, $\chi(d)= \chi_r(y)\chi_s(x)$ and by quadratic reciprocity
  \begin{equation*}
    \legendre{rs}{d}
    =
    (-1)^{\frac{x-1}{2}\frac{r-1}{2}} \legendre{y}{r} \legendre{s}{x}
    =
    \epsilon_x^{r-1} \legendre{y}{r} \legendre{s}{x}.
  \end{equation*}
	The sum $K_{\ell}(m,n;rs;\chi)$ now factors as
	\begin{equation*}
	  \sumstar_{y \, (r)} \overline{\chi_r(y)}
    \legendre{y}{r}
    e\left( \frac{m\overline{s}\overline{y}+ n\overline{s} y}{r} \right)
    \sumstar_{x \, (s)} \epsilon_x^{\ell+r-1} \overline{\chi_s(x)}
    \legendre{s}{x}
    e\left( \frac{m\overline{r}\overline{x}+ n\overline{r} x}{s} \right),
	\end{equation*}
  where the $\star$ in the sum indicated that we are only summing over residues relatively prime to
  the modulus, $\overline{y}$ modulo $r$ is such that $\overline{y}y \equiv 1 \bmod (r)$, and
  $\overline{x}$ modulo $s$ is such that $\overline{x}x \equiv 1 \bmod (s)$.
  This completes the proof.
\end{proof}

\begin{lemma}\label{lem:appendix-Salie-twist-mult}
  Let $r,s \in \mathbb{N}$ be two relatively prime integers with $N \mid rs$ and $v_2(rs) \neq 1$.
  Suppose the Dirichlet character $\chi$ modulo $N$ factors as $\chi_r$ modulo $(N,r)$ times
  $\chi_s$ modulo $(N,s)$ and $\overline{r},\overline{s} \in \mathbb{Z}$ are integers satisfying
  $\overline{r}r+\overline{s}s=1$. Then, we have
	\begin{equation}\label{eq:twist-Salie-twist-relation}
		S(m,n;rs;\chi) = S(m\overline{s},n\overline{s};r;\chi_r) S(m\overline{r},n\overline{r};s;\chi_s).
	\end{equation}
\end{lemma}

\begin{proof}
  The proof is analoguous, we just need to note that for $d=x\overline{r}r +y \overline{s}s$ as
  before we have
  \begin{equation*}
    \legendre{d}{rs} = \legendre{d}{r} \legendre{d}{s} = \legendre{y}{r} \legendre{x}{s}.%
    \qedhere   
  \end{equation*}
\end{proof}

\begin{lemma}\label{lem:appendix-Salie-bound}
  Let $p$ be an odd prime and $c=p^{\alpha}$.
  Further, let $\chi$ be a Dirichlet character of conductor $\mathfrak{c}_{\chi}=p^{\gamma}$ with
  $\gamma \le \alpha$.
  Then, for any two integers $m,n \in \mathbb{Z}$, we have
	\begin{equation*}
	\lvert S(m,n;c;\chi)\rvert
  \le
  \tau(c) (m,n,c)^{1/2} c^{1/2} \mathfrak{c}_{\chi}^{1/2}.
	\end{equation*}
\end{lemma}

\begin{proof} Applying~\cite[Thm 9.3]{KnightlyLi13} yields the bound
	\begin{equation*}
    \lvert S(m,n;c;\chi) \rvert \le \tau(c) (m,n,c)^{1/2} c^{1/2} \max\{\mathfrak{c}_{\chi},p\}^{1/2}.
	\end{equation*}
  Thus it only remains to consider when $\chi$ is trivial.
  Writing $p^{\delta}=(m,n,c)$,
	\begin{equation*}
    S(m,n;c;\chi)
    =
    p^{\delta}  S(m/p^{\delta},n/p^{\delta};c/p^{\delta}; \slegendre{\cdot}{p}^{\delta}).
	\end{equation*}
  Thus, we may assume $(m,n,c)=1$ from now on.
  If either $c \mid m$ or $c \mid n$, then $S(m,n;c;\legendre{\cdot}{p}^{\delta})$ is either a
  Ramanujan sum or a Gau{\ss} sum.
  In either case, we have
	\begin{equation*}
	\lvert S(m,n;c;\slegendre{\cdot}{p}^{\delta}) \rvert \le 2 c^{1/2}.
	\end{equation*}
	The remaining cases follow from~\cite[Propositions 9.4, 9.7, 9.8]{KnightlyLi13}.
\end{proof}

\begin{lemma}\label{lem:appendix-Kloos-2-bound}
  Let $c=2^{\alpha}$ with $\alpha \ge 2$ and $\ell \in \mathbb{Z}$ an odd integer.
  Further, let $\chi$ be a Dirichlet character of conductor $\mathfrak{c}_{\chi}=2^{\gamma}$ with
  $\gamma \le \alpha$.
  Then, for any two integers $m,n \in \mathbb{Z}$, we have
	\begin{equation*}
    \lvert K_{\ell}(m,n;c;\chi) \rvert
    \le
    4 \tau(c) (m,n,c)^{1/2} c^{1/2} \mathfrak{c}_{\chi}^{1/2}.
	\end{equation*}
\end{lemma}

\begin{proof}
  We expand $\epsilon_d^{\ell} = \frac{1+i^{\ell}}{2}+\frac{1-i^{\ell}}{2} \legendre{-1}{d}$ in
  terms of Dirichlet characters and apply~\cite[Theorem 9.3]{KnightlyLi13} to get
	\begin{equation*}
    \lvert K_{\ell}(m,n;c;\chi) \rvert
    \le
    \sqrt{2} \tau(c) (m,n,c)^{1/2} c^{1/2} \max\{\mathfrak{c}_{\chi},8\}^{1/2}.
    \qedhere
	\end{equation*}
\end{proof}

Combining the four previous lemmas gives Proposition~\ref{prop:appendix-Kloos-bound}.

\subsection*{Proof of Proposition~\ref{prop:appendix-oscill-int}}

We begin with some classical bounds on oscillatory integrals and Bessel functions.

\begin{lemma}\label{lem:appendix-2-derivative-est}
  Let $g(x)$ is a real valued and smooth function on an interval $(a,b)$, continuous at the
  endpoints, and with $\lvert g^{(k)}(x) \rvert \gg \lambda >0$ on $(a,b)$ for some integer $k\ge
  1$.
  In the case $k=1$, assume also that $g$ is monotonic.
  Furthermore, let $f(x)$ a continuously differentiable function on the closed interval $[a,b]$,
  then
	\begin{equation*}
	\int_a^b f(x) e^{i g(x)} dx
  \ll
  \lambda^{-1/k} \left( \lvert f(b) \rvert+\int_a^b \lvert f'(x) \rvert dx \right).
  \end{equation*}
\end{lemma}
\begin{proof}
	See for example~\cite[Chap.\@ 7 \S 1 Prop.\@ 2 \& Cor.]{Stein93}.
\end{proof}

\begin{lemma}\label{lem:appendix-BesselJ-bounds}
We have the following uniform bounds on the $J$-Bessel function for $q \in \mathbb{R}^+$:
	\begin{align*}
		\lvert J_{2it}(q) \rvert
    &\ll
    \cosh(\pi t) \cdot \min\{q^{-1/2}, 1+\lvert \log(q) \rvert\}
    \quad \forall t \in \mathbb{R},
    \\
		\lvert J_{2it}(q)-J_{-2it}(q) \rvert
    &\le
    \lvert \sinh(\pi t) \rvert \cdot \min\{q^{-1/2}, 1+\lvert \log(q) \rvert\}
    \quad \forall t \in \mathbb{R}.
	\end{align*}
\end{lemma}

\begin{proof}
	We shall use the integral representation (see~\cite[Eq.\@ (12) p.\@ 180]{WatsonBessel})
	\begin{equation}\label{eq:appendix-BesselJ-intrep}
		J_{v}(x) = \frac{2}{\pi} \int_0^{\infty} \sin(x \cosh(\xi)-\tfrac{\pi}{2}v) \cosh(v \xi) d\xi, \quad |\Re(v)|<1, x>0.
	\end{equation}
	It follows that
	\begin{align*}
			\lvert J_{2it}(q) \rvert  &\ll \cosh(\pi t) \sum_{\pm} \left\lvert  \int_0^{\infty} e^{\pm i q
      \cosh(\xi)} \cos(2t\xi) d\xi \right \rvert
      \\
			&\ll \cosh(\pi t) \sum_{\ell=0}^{\infty} \min\{1,q^{-1/2} e^{-\ell/2}\}
      \\
      &\ll \cosh(\pi t) \cdot \min\{q^{-1/2}, 1+\lvert \log(q) \rvert\},
	\end{align*}
  where we have used the second derivative test, Lemma~\ref{lem:appendix-2-derivative-est}, on the
  individual intervals $[\ell,\ell+1]$ for $\ell \in \mathbb{N}_0$.
  Likewise, we have
	\begin{align*}
			\lvert J_{2it}(q)-J_{-2it}(q) \rvert &\ll \lvert \sinh(\pi t) \rvert
      \cdot
      \left \lvert \int_0^{\infty} \cos(q \cosh(\xi)) \cos(2t \xi) d\xi \right \rvert
      \\
			&\ll \lvert \sinh(\pi t) \rvert \sum_{\ell} \min\{1,q^{-1/2} e^{-\ell/2}\}
      \\
			& \ll \lvert \sinh(\pi t) \rvert \cdot \min\{q^{-1/2}, 1+|\log(q)|\}.
      \qedhere
	\end{align*}
\end{proof}

In order to prove Proposition~\ref{prop:appendix-oscill-int}, we begin with an
alternate integral representation of $I_{\kappa}(\omega,t)$. For $\kappa > 0$,
we have~\cite[Eq.\@ (44)]{Proskurin05}\footnote{Correcting a typo in the integral
limits.}
\begin{align*}
	I_{\kappa}(\omega,t) =&
  \frac{2\pi i \omega^{1-\kappa}}{\sinh(2 \pi t)}
  \int_0^{\omega}
  \begin{aligned}[t]
  \Big(
    &J_{2it}(q) \cos\left(\tfrac{\pi}{2} (1-\kappa-2it)\right) \\
    &-
    J_{-2it}(q) \cos\left(\tfrac{\pi}{2} (1-\kappa+2it)\right)
  \Big) q^{\kappa} \frac{dq}{q}
  \end{aligned}
  \\
	=& -\frac{2\pi  \omega^{1-\kappa} \sinh(\pi t)}{\sinh(2 \pi t)}
  \cos(\pi \tfrac{\kappa}{2}) \int_0^{\omega}
  \left(J_{2it}(q) + J_{-2it}(q) \right)
  q^{\kappa} \frac{dq}{q}
  \\
	& + \frac{2\pi i \omega^{1-\kappa} \cosh(\pi t)}{\sinh(2 \pi t)}
  \sin(\pi \tfrac{\kappa}{2})
  \int_0^{\omega} \left(J_{2it}(q)- J_{-2it}(q) \right)
  q^{\kappa} \frac{dq}{q},
\end{align*}
and for $\kappa \le 0$, we have the similar expression
\begin{align*}
	I_{\kappa}(\omega,t)
  =&
  -\frac{2\pi i \omega^{1-\kappa}}{\sinh(2 \pi t)}
  \int_{\omega}^{\infty}
  \begin{aligned}[t]
  \Big(
    &J_{2it}(q) \cos\left(\tfrac{\pi}{2} (1-\kappa-2it)\right)
    \\
    &- J_{-2it}(q) \cos\left(\tfrac{\pi}{2} (1-\kappa+2it)\right)
  \Big)
  q^{\kappa} \frac{dq}{q}
  \end{aligned}
  \\
	=& \frac{2\pi  \omega^{1-\kappa} \sinh(\pi t)}{\sinh(2 \pi t)} \cos(\pi \tfrac{\kappa}{2})
  \int_{\omega}^{\infty} \left(J_{2it}(q) + J_{-2it}(q) \right)    q^{\kappa} \frac{dq}{q}
  \\
	&- \frac{2\pi i \omega^{1-\kappa} \cosh(\pi t)}{\sinh(2 \pi t)} \sin(\pi \tfrac{\kappa}{2})
  \int_{\omega}^{\infty} \left(J_{2it}(q)- J_{-2it}(q) \right)   q^{\kappa} \frac{dq}{q},
\end{align*}
which is proved by contour shifts and Bessel function relations analogous to~\cite[Lemma
3.3]{AhlgrenAndersen18}, where the case $\kappa=-\frac{1}{2}$ was treated.
We tackle first the case $\kappa \in ]0,2[$ and then the case $\kappa \in ]-2,0[$ with the case
$\kappa=0$ already having been treated by Kuznetsov~\cite[Eq. (5.15)]{Kuznetsov80}.
Thus, let $\kappa \in ]0,2[$.
In the light of Lemma~\ref{lem:appendix-BesselJ-bounds}, the double integral
\begin{equation*}
  \int_0^T t I_{\kappa}(\omega,t)dt
\end{equation*}
converges absolutely and may thus be computed as
\begin{equation*}
  \lim_{\epsilon \to 0^+} G_{\kappa}^{\epsilon}(\omega,T)
  =
  \lim_{\epsilon \to 0^+}\int_0^T t I^{\epsilon}_{\kappa}(\omega,t)dt,
\end{equation*}
where
\begin{align*}
	I^{\epsilon}_{\kappa}(\omega,t)
	=&
  -\frac{2\pi  \omega^{1-\kappa} \sinh(\pi t)}{\sinh(2 \pi t)} \cos(\pi \tfrac{\kappa}{2})
  \int_{\epsilon}^{\omega} \left(J_{2it}(q) + J_{-2it}(q) \right)    q^{\kappa} \frac{dq}{q}
  \\
	&+ \frac{2\pi i \omega^{1-\kappa} \cosh(\pi t)}{\sinh(2 \pi t)} \sin(\pi \tfrac{\kappa}{2})
  \int_{\epsilon}^{\omega} \left(J_{2it}(q)- J_{-2it}(q) \right)   q^{\kappa} \frac{dq}{q}.
\end{align*}
We now make use of the integral representation~\eqref{eq:appendix-BesselJ-intrep} to write
\begin{align*}
	I^{\epsilon}_{\kappa}(\omega,t)
  =&
  -4  \omega^{1-\kappa} \cos(\pi \tfrac{\kappa}{2}) \int_{\epsilon}^{\omega}
  \int_0^{\infty} \sin(q \cosh(\xi)) \cos(2t \xi) d\xi    q^{\kappa} \frac{dq}{q}
  \\
	&+ 4 \omega^{1-\kappa} \sin(\pi \tfrac{\kappa}{2}) \int_{\epsilon}^{\omega}
  \int_0^{\infty} \cos(q \cosh(\xi)) \cos(2t \xi) d\xi   q^{\kappa} \frac{dq}{q}.
\end{align*}
We truncate the inner integral and use Lemma~\ref{lem:appendix-2-derivative-est} to bound the remainder.
This yields
\begin{align}
	\int_0^{\infty} &\cos(q \cosh(\xi)) \cos(2t \xi) d\xi
  \\
	=& \int_0^{A} \cos(q \cosh(\xi)) \cos(2t \xi) d\xi + O(q^{-1/2}e^{-A/2})
  \label{eq:appendix-innermost-truncation}
  \\
	=&\begin{aligned}[t]
    \frac{\sin(2At)}{2t}\cos(q \cosh(A))
	  &+ \frac{q}{2t} \int_0^A \sin(q \cosh(\xi)) \sinh(\xi) \sin(2 t \xi) d \xi
    \\
    &+ O(q^{-1/2}e^{-A/2}),
  \end{aligned}
\end{align}
after further integration by parts. Hence,
\begin{align*}
		\int_0^T t \int_{\epsilon}^{\omega}
    \int_0^{\infty} &\cos(q \cosh(\xi)) \cos(2t \xi) d\xi  q^{\kappa} \frac{dq}{q} dt
    \\
		=& \frac{1-\cos(2AT)}{4A} \int_{\epsilon}^{\omega} \cos(q \cosh(A)) q^{\kappa-1} dq
    \\
		&+ \frac{1}{4} \int_0^A \frac{\sinh(\xi)}{\xi} (1-\cos(2T \xi))
    \int_{\epsilon}^{\omega} \sin(q \cosh(\xi)) q^{\kappa} dq   d\xi
    \\
    &+ O(T\omega^{\kappa+1/2}e^{-A/2}).
\end{align*}
After taking $A \to \infty$, this simplifies to
\begin{align*}
	\int_0^T t  \int_{\epsilon}^{\omega}
  \int_0^{\infty} &\cos(q \cosh(\xi)) \cos(2t \xi) d\xi q^{\kappa-1} dq dt
  \\
	&= \frac{1}{4} \int_0^{\infty} \frac{\sinh(\xi)}{\xi} (1-\cos(2T \xi))
  \int_{\epsilon}^{\omega} \sin(q \cosh(\xi)) q^{\kappa} dq   d\xi.
\end{align*}
Analogously, we find
\begin{align*}
	\int_0^T t  \int_{\epsilon}^{\omega}
  \int_0^{\infty} &\sin(q \cosh(\xi)) \cos(2t \xi) d\xi q^{\kappa-1} dq dt
  \\
	&= - \frac{1}{4} \int_0^{\infty} \frac{\sinh(\xi)}{\xi} (1-\cos(2T \xi))
  \int_{\epsilon}^{\omega} \cos(q \cosh(\xi)) q^{\kappa} dq   d\xi.
\end{align*}
We conclude that
\begin{equation}\label{eq:appendix-G-kappa-eps-simple}
	G_{\kappa}^{\epsilon}(\omega,T)
  =
  \omega^{1-\kappa} \int_0^{\infty} \frac{\sinh(\xi)}{\xi} (1-\cos(2T\xi))
  \int_{\epsilon}^{\omega} \cos(q \cosh(\xi)-\pi \tfrac{\kappa}{2}) q^{\kappa} dq d\xi.
\end{equation}

We now split the integral $\int_{\epsilon}^{\omega}$ as $\int_{0}^{\omega}-\int_{0}^{\epsilon}$ and
consequently write
$G_{\kappa}^{\epsilon}(\omega,T) = H_{\kappa}^{\omega}(\omega,T)-H_{\kappa}^{\epsilon}(\omega,T)$,
where
\begin{equation}\label{eq:H-kappa-eta}
	H_{\kappa}^{\eta}(\omega,T)
  =
  \omega^{1-\kappa} \int_0^{\infty} \frac{\sinh(\xi)}{\xi} (1-\cos(2T\xi))
  \int_{0}^{\eta} \cos(q \cosh(\xi)-\pi \tfrac{\kappa}{2}) q^{\kappa} dq d\xi.
\end{equation}
In due course, we shall see that this man{\oe}uvre is legitimate. 
Integration by parts and a further substitution shows
\begin{equation}\label{eq:appendix-inner-int-exp}
  \begin{split}
	\int_0^{\eta} \cos(q \cosh(\xi)-\pi \tfrac{\kappa}{2}) q^{\kappa} dq
	= \frac{1}{\cosh(\xi)}\eta^{\kappa}\sin(\eta \cosh(\xi)-\pi \tfrac{\kappa}{2})
	\\- \frac{\kappa}{\cosh(\xi)^{1+\kappa}} \int_0^{\eta \cosh(\xi)} \sin(q-\pi\tfrac{\kappa}{2}) q^{\kappa-1} dq.
  \end{split}
\end{equation}
The contribution of the former term to $H^{\eta}_{\kappa}(\omega,T)$ is
\begin{multline*}
	\omega^{1-\kappa}\eta^{\kappa} \int_0^{\infty} \frac{\tanh(\xi)}{\xi}(1-\cos(2T\xi))\sin(\eta \cosh(\xi)-\pi \tfrac{\kappa}{2}) d\xi \\
	\ll \omega^{1-\kappa} \eta^{\kappa} \sum_{\ell=0}^{\infty} \frac{1}{\ell+1} \min\{1, \eta^{-1/2} e^{-\ell/2} \}
	\ll
  \begin{cases}
    \omega^{1-\kappa} \eta^{\kappa-\frac{1}{2}}, & \eta \ge 1,
    \\
    \omega^{1-\kappa}\eta^{\kappa} (1+\lvert \log(\eta) \rvert), & \eta \le 1,
  \end{cases}
\end{multline*}
where we have use the second derivative test, Lemma~\ref{lem:appendix-2-derivative-est}, to each of
the intervals $[\ell,\ell+1]$ individually.
For $\eta=\epsilon$, this vanishes in the limit $\epsilon \to 0^+$ and for $\eta=\omega$, this
contribution is sufficient.
It remains to deal with the secondary term of~\eqref{eq:appendix-inner-int-exp}.
We find trivially that
\begin{equation}\label{eq:appendix-second-triv}
	\left\lvert
    \frac{\kappa}{\cosh(\xi)^{1+\kappa}} \int_0^{\eta \cosh(\xi)} \sin(q-\pi\tfrac{\kappa}{2}) q^{\kappa-1} dq
  \right\rvert \ll \frac{\eta^{\kappa}}{\cosh(\xi)}.
\end{equation}

We now restrict ourselves to the case $1 \le \kappa < 2$. In this case, we can integrate by parts once more to find
\begin{equation}\label{eq:second-kappa-large}
	\left\lvert
    \frac{\kappa}{\cosh(\xi)^{1+\kappa}} \int_0^{\eta \cosh(\xi)}
    \sin(q-\pi\tfrac{\kappa}{2}) q^{\kappa-1} dq
  \right \rvert
  \ll \frac{1}{\cosh(\xi)^{1+\kappa}}+\frac{\eta^{\kappa-1}}{\cosh(\xi)^2}.
\end{equation}
Hence, the contribution to $H^{\eta}_{\kappa}(\omega,T)$ for $1 \le \kappa <2$ is bounded by
\begin{multline*}
	\omega^{1-\kappa}\int_0^{\infty} \frac{\tanh(\xi)}{\xi} \min\left\{\eta^{\kappa},\frac{1}{\cosh(\xi)^{\kappa}}+\frac{\eta^{\kappa-1}}{\cosh(\xi)}\right\} d\xi \\ \ll
  \begin{cases}
    \omega^{1-\kappa}(1+\eta^{\kappa-1}), & \eta \ge 1,
    \\
    \omega^{1-\kappa}\eta^{\kappa}(1+|\log(\eta)|), & \eta\le 1.
  \end{cases}
\end{multline*}
For $\eta = \epsilon$, this vanishes in the limit and for $\eta= \omega$, this is sufficient.

We now turn to the case $0 < \kappa < 1$.
Here, we complete the integral, note that the completed integral vanishes (see~\cite[Eq.
3.712]{GradshteynRyzhik07}), and used the first derivative test,
Lemma~\ref{lem:appendix-2-derivative-est}, to bound the remainder.
\begin{align}
		\int_0^{\eta \cosh(\xi)} &\sin(q-\pi \tfrac{\kappa}{2}) q^{\kappa-1} dq
    \\
    =& \int_0^{\infty} \sin(q-\pi \tfrac{\kappa}{2}) q^{\kappa-1} dq-\int_{\eta \cosh(\xi)}^{\infty} \sin(q-\pi \tfrac{\kappa}{2}) q^{\kappa-1} dq
    \\
		=& 0 + O\left( \eta^{\kappa-1} \cosh(\xi)^{\kappa-1}  \right).
    \label{eq:appendix-second-int-completion}
\end{align}
Hence, the contribution to $H^{\eta}_{\kappa}(\omega,T)$ for $0 < \kappa <1$ is bounded by
\begin{equation*}
  \omega^{1-\kappa}\int_0^{\infty} \frac{\tanh(\xi)}{\xi}
  \min\left\{\eta^{\kappa},\frac{\eta^{\kappa-1}}{\cosh(\xi)}\right\} d\xi
  \ll \begin{cases}
    \omega^{1-\kappa}\eta^{\kappa-1}, & \eta \ge 1,
    \\
    \omega^{1-\kappa}\eta^{\kappa}(1+|\log(\eta)|), & \eta\le 1.
  \end{cases}
\end{equation*}
For $\eta = \epsilon$, this vanishes in the limit and for $\eta= \omega$, this is sufficient.
This concludes the proof for $\kappa \in ]0,2[$.

Let us now assume that $\kappa \in ]-2,0[$.
We argue as before and find that $G_{\kappa}(\omega,T)$ may be computed as the limit
$\displaystyle \lim_{\Omega \to \infty} G_{\kappa}^{\Omega}(\omega,T)
=
\lim_{\Omega \to \infty} \int_0^T t I_{\kappa}^{\Omega}(\omega,t) dt$,
where
\begin{align*}
	I_{\kappa}^{\Omega}(\omega,t)
  =& 4  \omega^{1-\kappa} \cos(\pi \tfrac{\kappa}{2}) \int_{\omega}^{\Omega}
  \int_0^{\infty} \sin(q \cosh(\xi)) \cos(2t \xi) d\xi    q^{\kappa} \frac{dq}{q}
  \\
	&- 4 \omega^{1-\kappa} \sin(\pi \tfrac{\kappa}{2}) \int_{\omega}^{\Omega}
  \int_0^{\infty} \cos(q \cosh(\xi)) \cos(2t \xi) d\xi   q^{\kappa} \frac{dq}{q}.
\end{align*}
Repeating the steps~\eqref{eq:appendix-innermost-truncation}-\eqref{eq:appendix-G-kappa-eps-simple}, we find
\begin{equation}\label{eq:appendix-G-kappa-Omega}
	G_{\kappa}^{\Omega}(\omega,T)
  =
  -\omega^{1-\kappa} \int_0^{\infty} \frac{\sinh(\xi)}{\xi} (1-\cos(2T\xi))
  \int_{\omega}^{\Omega} \cos(q \cosh(\xi)-\pi \tfrac{\kappa}{2}) q^{\kappa} dq d\xi.
\end{equation}
We write this again as a difference $H_{\kappa}^{\Omega}(\omega,T)-H_{\kappa}^{\omega}(\omega,T)$, where
\begin{equation}\label{eq:appendix-H-Omega-def}
	H_{\kappa}^{\eta}(\omega,T)
  =
  \omega^{1-\kappa} \int_0^{\infty} \frac{\sinh(\xi)}{\xi} (1-\cos(2T\xi))
  \int_{\eta}^{\infty} \cos(q \cosh(\xi)-\pi \tfrac{\kappa}{2}) q^{\kappa} dq d\xi.
\end{equation}
By integration by parts and subsequent substitution, we have
\begin{multline}\label{eq:appendix-inner-Omega-int-by-parts}
	\int_{\eta}^{\infty} \cos(q \cosh(\xi)-\pi \tfrac{\kappa}{2}) q^{\kappa} dq d\xi =  -\frac{1}{\cosh(\xi)}\sin(\eta \cosh(\xi)-\pi \tfrac{\kappa}{2})\eta^{\kappa} \\
	- \frac{\kappa}{\cosh(\xi)^{1+\kappa}} \int_{\eta \cosh(\xi)}^{\infty} \sin(q -\pi \tfrac{\kappa}{2}) q^{\kappa-1} dq.
\end{multline}
The first term contributes
\begin{multline*}
	-\omega^{1-\kappa}\eta^{\kappa}\int_0^{\infty} \frac{\tanh(\xi)}{\xi}(1-\cos(2T\xi))\sin(\eta \cosh(\xi)-\pi \tfrac{\kappa}{2}) d\xi \\
	\ll \omega^{1-\kappa}\eta^{\kappa} \sum_{\ell=0}^{\infty} \frac{1}{\ell+1} \min\{1,\eta^{-1/2}e^{-\ell/2}\} \ll
  \begin{cases}
    \omega^{1-\kappa} \eta^{\kappa-\frac{1}{2}}, & \eta \ge 1,
    \\
    \omega^{1-\kappa}\eta^{\kappa} (1+|\log(\eta)|), & \eta \le 1
  \end{cases}
\end{multline*}
to $H_{\kappa}^{\eta}(\omega,T)$, where we have used the second derivative test, Lemma \ref{lem:appendix-2-derivative-est}, on each of the intervals $[\ell,\ell+1]$. For $\eta= \Omega$, this vanishes in the limit $\Omega \to \infty$ and for $\eta= \omega$, the contribution is sufficient.
The secondary term in \eqref{eq:appendix-inner-Omega-int-by-parts} we may bound trivially or using
Lemma \ref{lem:appendix-2-derivative-est}, yielding
$$
\left| - \frac{\kappa}{\cosh(\xi)^{1+\kappa}} \int_{\eta \cosh(\xi)}^{\infty} \sin(q -\pi \tfrac{\kappa}{2}) q^{\kappa-1} dq \right| \ll \min\left\{ \frac{\eta^{\kappa}}{\cosh(\xi)}, \frac{\eta^{\kappa-1}}{\cosh(\xi)^2} \right\}.
$$
The contribution from the secondary term to $H_{\kappa}^{\eta}(\omega,T)$ is thus bounded by
$$
\omega^{1-\kappa}\eta^{\kappa} \int_0^{\infty} \frac{\tanh(\xi)}{\xi} \min\left\{ 1, \frac{1}{\eta \cosh(\xi)} \right\} d\xi \ll \begin{cases} \omega^{1-\kappa}\eta^{\kappa-1}, & \eta \ge 1, \\ \omega^{1-\kappa}\eta^{\kappa}(1+|\log(\eta)|), & \eta\le 1. \end{cases}
$$
For $\eta= \Omega$, this vanishes in the limit $\Omega \to \infty$ and for $\eta = \omega$, this gives a sufficient contribution, thereby completing the proof.

\bibliographystyle{alpha}
\bibliography{compiled_bibliography.bib}

\end{document}